\documentclass[11pt]{amsart}
\usepackage{pgf,tikz,pgfplots,color}
\pgfplotsset{compat=1.15}
\usepackage{mathrsfs} 
\usetikzlibrary{arrows}
\usepackage{fancyhdr}
\usepackage{lastpage} 
\usepackage{mathtools}

\usepackage[margin=1in]{geometry}
\usepackage{geometry}
\usepackage[toc,page]{appendix}
\usepackage[utf8]{inputenc}
\usepackage{hyperref}
\hypersetup{
	colorlinks=true,
	linkcolor=blue,
	citecolor=brown,
	filecolor=magenta, 
	urlcolor=blue,
}
\usepackage[notcite,notref]{}        
\usepackage{amsmath}
\usepackage{amsfonts}
\usepackage{amssymb}
\usepackage{amsthm}
\usepackage{setspace}
\usepackage{inputenc}
\usepackage[english]{babel} 
\usepackage{comment}
\usepackage[shortlabels]{enumitem}
\usepackage[T1]{fontenc}
\usepackage{esint}

\numberwithin{equation}{section}

\makeatletter
\def\@tocline#1#2#3#4#5#6#7{\relax
  \ifnum #1>\c@tocdepth 
  \else
    \par \addpenalty\@secpenalty\addvspace{#2}%
    \begingroup \hyphenpenalty\@M
    \@ifempty{#4}{%
      \@tempdima\csname r@tocindent\number#1\endcsname\relax
    }{%
      \@tempdima#4\relax
    }%
    \parindent\z@ \leftskip#3\relax \advance\leftskip\@tempdima\relax
    \rightskip\@pnumwidth plus4em \parfillskip-\@pnumwidth
    #5\leavevmode\hskip-\@tempdima
      \ifcase #1
       \or\or \hskip 1em \or \hskip 2em \else \hskip 3em \fi%
      #6\nobreak\relax
    \hfill\hbox to\@pnumwidth{\@tocpagenum{#7}}\par
    \nobreak
    \endgroup
  \fi}
\makeatother

\title[]{Sobolev and H\"older estimates for the $\overline \partial$  equation on pseudoconvex domains of finite type in $\mathbb C^2$}            

\author[]{Ziming Shi}

\address{Department of Mathematics,
	University of California-Irvine, Irvine, CA, 92697} 
\email{zimings3@uci.edu}


\newcommand{\dist}{\operatorname{dist}}

\newcommand{\supp}{\operatorname{supp}}

\newcommand{\Ker}{\operatorname{Ker}}
\newcommand{\lip}{\operatorname{Lip}}
\newcommand{\loc}{\mathrm{loc}}

\newtheorem{thm}{Theorem}[section]
\newtheorem{cor}[thm]{Corollary} 
\newtheorem{prop}[thm]{Proposition}
\newtheorem{lemma}[thm]{Lemma}

\theoremstyle{definition}
\newtheorem{defn}[thm]{Definition}
\newtheorem{exmp}[thm]{Example}
\newtheorem{ques}[thm]{Question}

\theoremstyle{remark}
\newtheorem{rem}[thm]{Remark}
\newtheorem*{clm}{Claim}
\newtheorem*{ack}{Acknowledgment}

\renewcommand{\th}[1]{\begin{thm}\label{#1}}
	\renewcommand{\eth}{\end{thm}}
\newcommand{\co}[1]{\begin{cor}\label{#1}}
	\newcommand{\eco}{\end{cor}}
\newcommand{\pr}[1]{\begin{prop}\label{#1}}
	\newcommand{\epr}{\end{prop}}

\newcommand{\df}[1]{\begin{defn}\label{#1}}
	\newcommand{\edf}{\end{defn}}
\newcommand{\ex}[1]{\begin{exmp}\label{#1}} 
	\newcommand{\eex}{\end{exmp}}
\newcommand{\qu}[1]{\begin{ques}\label{#1}}
	\newcommand{\equ}{\end{ques}}  
\newcommand{\mk}{\begin{rem}}
	\newcommand{\emk}{\end{rem}}
\newcommand{\cl}{\begin{clm}}
	\newcommand{\ecl}{\end{clm}} 
\newcommand{\ac}{\begin{ack}}
	\newcommand{\eac}{\end{ack}} 

\newcommand{\ga}{\begin{gather}}
\newcommand{\ega}{\end{gather}}
\newcommand{\gan}{\begin{gather*}}
\newcommand{\egan}{\end{gather*}}
\newcommand{\al}{\begin{gngn}}
	\newcommand{\eal}{\end{align}}
\newcommand{\aln}{\begin{align*}}
\newcommand{\ealn}{\end{align*}}
\newcommand{\eq}[1]{\begin{equation}\label{#1}}
\newcommand{\eeq}{\end{equation}}

\newcommand{\pa}{\partial{}}
\newcommand{\na}{\nabla}
\newcommand{\db}{\dbar}

\newcommand{\we}{\wedge}

\newcommand{\ra}{\longrightarrow}

\newcommand{\bl}{\bigl(} 
\newcommand{\br}{\bigr)}  
\newcommand{\Bl}{\Bigl(} 
\newcommand{\Br}{\Bigr)}  
\newcommand{\blb}{\bigl[} 
\newcommand{\brb}{\bigr]}   
\newcommand{\Blb}{\Bigl[}
\newcommand{\Brb}{\Bigr]} 
\newcommand{\blc}{\bigl\{ } 
\newcommand{\brc}{\bigr\} }

\newcommand{\bn}{\big\|}   
\newcommand{\Bn}{\Big\|}  

\newcommand{\sm}{\setminus}

\newcommand{\pp}[2]{\frac{\partial #1}{\partial #2}}

\newcommand{\Z}{\mathbb{Z}}
\newcommand{\R}{\mathbb{R}} 
\newcommand{\C}{\mathbb{C}}

\newcommand{\N}{\mathbb{N}}

\newcommand{\V}{\mathcal{V}}
\newcommand{\U}{\mathcal{U}} 

\newcommand{\mc}{\mathcal}

\newcommand{\1}{\mathbf{1}}

\newcommand{\ov}{\overline}
\newcommand{\ti}{\tilde}
\newcommand{\wti}{\widetilde}
\newcommand{\hht}{\widehat}
\newcommand{\mr}{\mathring}

\newcommand{\RE}{\operatorname{Re}}
\newcommand{\IM}{\operatorname{Im}}
\newcommand{\dbar}{\overline\partial}

\newcommand{\all}{\alpha}

\newcommand{\del}{\delta}
\newcommand{\Del}{\Delta}
\newcommand{\var}{\varphi}
\newcommand{\e}{\epsilon}
\newcommand{\ve}{\varepsilon}
\newcommand{\om}{\omega}
\newcommand{\Om}{\Omega}
\newcommand{\thh}{\theta}
 
\newcommand{\La}{\Lambda}
\newcommand{\la}{\lambda}

\newcommand{\gm}{\gamma}
\newcommand{\Gm}{\Gamma}

\newcommand{\si}{\sigma}

\newcommand{\yh}{\frac{1}{2}}
\newcommand{\yt}{\frac{1}{3}}

\newcommand{\re}[1]{(\ref{#1})}

\newcommand{\rl}[1]{Lemma~\ref{#1}}
\newcommand{\rc}[1]{Corollary~\ref{#1}}
\newcommand{\rp}[1]{Proposition~\ref{#1}}
\newcommand{\rt}[1]{Theorem~\ref{#1}}
\newcommand{\rd}[1]{Definition~\ref{#1}}

\newcounter{pp}
\newcommand{\bpp}{\begin{list}{$\hspace{-1em}\alph{pp})$}{\usecounter{pp}}}
	\newcommand{\epp}{\end{list}}

\newcounter{ppp}
\newcommand{\bppp}{\begin{list}{$\hspace{-1em}(\roman{ppp})$}{\usecounter{ppp}}}
	\newcommand{\eppp}{\end{list}}

\newcommand{\Cc}{\mathcal{C}}

\newcommand{\Bs}{\mathscr{B}}
\newcommand{\Cf}{\mathfrak{C}}

\newcommand{\Ec}{\mathcal{E}}

\newcommand{\Fs}{\mathscr{F}}
\newcommand{\Gc}{\mathcal{G}}

\newcommand{\Hc}{\mathcal{H}}
\newcommand{\Kb}{\mathbb{K}}
\newcommand{\Kc}{\mathcal{K}}

\newcommand{\Nc}{\mathcal{N}} 
\newcommand{\Oc}{\mathcal{O}}

\newcommand{\Ss}{\mathscr{S}}

\newcommand{\Uc}{\mathcal{U}}

\begin{document} 
	\definecolor{rvwvcq}{rgb}{0.08235294117647059,0.396078431372549,0.7529411764705882}
\maketitle  

\begin{abstract} 
   We prove a homotopy formula which yields almost sharp estimates in all (positive-indexed) Sobolev and H\"older-Zygmund spaces for the $\overline \partial$ equation on pseudoconvex domains of finite type in $\mathbb C^2$, extending the earlier results of Fefferman-Kohn (1988), Range (1990), and Chang-Nagel-Stein (1992). 
The main novelty of our proof is the construction of holomorphic support functions that admit precise estimates when the parameter variable lies in a thin shell outside the domain. 
\end{abstract} 
\vspace{2cm}

\tableofcontents 

\section{Introduction} 
The goal of the present paper is the following:  
\begin{thm} \label{Thm::main} 
   Let $D \subset \C^2$ be a ($C^\infty$-)smoothly bounded pseudoconvex domain of finite type $m$. For each $\eta>0$, there exist linear operators $\Hc^\eta_i$, $i=1,2$ such that 
   \begin{enumerate}[(i)] 
    \item 
    $\Hc^\eta_i: H^{s,p}_{(0,i)} (D) \to H^{s+\frac{1}{m}-\eta,p}_{(0,i-1)}(D)$, for any $1<p<\infty$ and $s>\frac1p$. 
    \item 
   $\Hc^\eta_i: \La^s_{(0,i)}(D) \to \La^{s+\frac{1}{m}-\eta}_{(0,i-1)}(D)$ for any $s>0$. 
   \item 
  Suppose $\var \in H^{s,p}_{(0,1)}(D)$ and $\db \var \in H^{s,p}_{(0,2)}(D)$ (resp. $\var \in \La^s_{(0,1)}(\ov D)$ and $\db \var \in \La^s_{(0,2)}(\ov D)$), for $s,p$ given as above. Then 
  \[  
  \var = \db \Hc_1^\eta \var + \Hc_2^\eta \db \var
  \] 
in the sense of distributions. In particular $\Hc_1^\eta \var$ is a solution to the equation $\db u = \var$ for any $\var \in H^{s,p}_{(0,1)}( D)$ (or $\La^s(D) )$ with $\db \var = 0$. 
\end{enumerate} 
Here $H^{s,p}(D)$ is the fractional Sobolev space (see \rd{Def::Sobolev_Dom}), and $\La^s(D) $ is the H\"older-Zygmund space (see \rd{Def::H-Z}). $H^{s,p}_{(0,i)}(D)$ ($i=1,2$) denotes the space of $(0,i)$ forms with $H^{s,p}(D)$ coefficients, and similarly for $\La^s_{(0,i)}(D)$. 
\end{thm}

The study of global existence and regularity of the $\dbar$-equation on pseudoconvex domains is a fundamental problem in several complex variables. Since the early 1960s, two parallel schools of research have developed to solve the $\db$ equation. The first one is by solving the $\db$-Neumann problem. On any pseudoconvex domain of $C^\infty$ boundary, one can define the $L^2$ canonical solution $\db^\ast \Nc \var$, where $\Nc$ is the operator that solves the $\db$-Neumann boundary value problem. The solution $\db^\ast \Nc \var$ is called the canonical solution since it is the unique solution which is orthogonal to $\Ker(\db)$ under the $L^2$ inner product. The first global regularity result was obtained by Kohn \cite{Koh64}, who showed that on a strongly pseudoconvex domain with $C^\infty$ boundary, $\db^\ast \Nc$ is a bounded operator from $H^{s,2}(\Om)$ to $H^{s+\yh,2}(\Om)$ for $s \geq 0$. Later on in their monograph \cite{G-S77}, Greiner and Stein proves the boundedness $\db^\ast \Nc: H^{s,p}_{(0,1)}(\Om) \to H^{s+\yh,p}(\Om)$ for $s \geq 0$, and $\db^\ast \Nc: \La^r_{(0,1)}(\Om) \to \La^{r+\yh}(\Om)$ for $r>0$. Chang \cite{Cha89} extends the 
(Sobolev space) result of Greiner-Stein to any $(p,q)$ forms.


The second approach for solving the $\db$ equation -- the one which we will follow in this paper -- is to look for solution operators in the form of integral formula. This can be viewed as the higher dimensional generalizations of the Cauchy-Green operator which solves the $\db$ equation in $\C$. The method of integral formula has several advantages. It is more geometric in nature; requires less boundary regularity; and gives $L^\infty$ estimate that is not accessible by the $\db$-Neumann method. The theory was 
pioneered by Grauert and Lieb, and independently, by Henkin, and has been developed extensively on many classes of pseudoconvex domains. We mention a few notable works in this direction. 

On a strictly pseudoconvex domain with $C^2$ boundary, Henkin and Ramirez \cite{H-R71} constructed a solution operator that is bounded from $C^0(\ov \Om)$ to $C^{\yh}(\ov \Om)$. Assuming the boundary is $C^{k+2}$ for positive integers $k$, Siu \cite{Siu74} and Lieb-Range \cite{L-R80} constructed solution operators that is bounded from  $C^{k}(\ov \Om)$ to $C^{k+\yh}(\ov \Om)$. More recently, Gong \cite{Gong19} showed that under the minimal smoothness of $C^2$ boundary, there exists a solution operator that is bounded from $\La^r(\ov \Om)$ to $\La^{r+\yh}(\ov \Om)$, for any $r >1$. In \cite{S-Y24_1}, the authors constructed a solution operator which is bounded from $H^{s,p}(\ov \Om) \to H^{s+\yh,p}(\ov \Om)$ for any $s \in \R$ and $1<p<\infty$, under the assumption that $b\Om \in C^\infty$. For convex domains of finite type, sharp 
estimates have been obtained by \cite{D-F-F99}, \cite{Ale06} and \cite{Yao24}. 

Pseudoconvex domains of finite type in $\C^2$ was introduced by Kohn \cite{Koh72}, as a natural generalization of strict pseudoconvexity. Kohn showed that if the domain has type $m$ (which must be even due to pseudoconvexity), then the $\db$-Neumann problem is subelliptic of order $1/m$, from which the boundedness of the operator $\db^\ast \Nc: H^{s,2} \to H^{s+\frac{1}{m},2}$ follows. Using microlocal analysis, Fefferman and Kohn \cite{F-K88} proved the boundedness $\db^\ast \Nc: C^s(\ov \Om)  \to C^{s+1/m}(\ov \Om) $ for any $s>0$ such that $s+1/m$ is not an integer, and they also proved the sup norm estimate $\db^\ast \Nc: L^\infty(\Om) \to L^\infty(\Om)$. Soon after, Chang-Nagel-Stein \cite{C-N-S92} used method similar to the ones in \cite{G-S77} to show that $\db^\ast \Nc$ is bounded from $\La^s(\ov \Om)$ to $\La^{s+\frac{1}{m}}(\ov \Om)$ for any $s>0$. They also proved for any smooth complex tangential vector fields $L_1$ the boundedness of the operator $L_1 \db^\ast \Nc, \ov{L_1} \db^\ast \Nc: W^{k,p}(\Om) \to W^{k,p}(\Om)$ for any $1<p<\infty$ and non-negative integer $k$. 

The method of integral formula also plays an important role in the study of $\db$ equation on 
pseudoconvex domains of finite type in $\C^2$. As in the case of strict pseudoconvexity, the success of this method depends largely on solving a holomorphic division problem 
\[
  h_1(z,\zeta) (z_1 - \zeta_1) + h_2(z,\zeta) (z_2 - \zeta_2) = \Phi(z,\zeta). 
\]
Here the functions $h_i$ are called \emph{Leray maps} and has the important property that they are holomorphic in $z \in \Om$. The function $\Phi$ -- called the \emph{holomorphic support function} -- is non-vanishing for $z \in \Om$ and $\zeta \in b\Om$. Kohn and Nirenberg \cite{K-N73} constructed an example -- a finite type pseudoconvex domain on which such function $\Phi$ does not exist while vanishing at $z = \zeta \in b\Om$. Thus in general one cannot expect to obtain $h_1,h_2$ that are bounded uniformly for $(z,\zeta) \in (\Om,b\Om)$. Nevertheless, on a class of domains in $\C^2$ which include the Kohn-Nirenberg example, Fornaess \cite{For86} proved the existence of $h_1,h_2$ with $\Phi \equiv 1$ and $h_i$, $i=1,2$ satisfying a certain weighted integral estimate, which then allows him to prove the sup-norm estimate for the $\db$ equation on such domains. Belanger \cite{Bel93} then refined Fornaess' method and obtained some (far from optimal) H\"older estimates for $\db$ on the same class of domains. 

In \cite{Ran90}, Range constructed for each $\eta >0$ an integral solution operator $T^\eta$ that is bounded from $L^\infty(\ov \Om)$ to $C^{\frac{1}{m}-\eta}(\ov \Om)$, on any pseudoconvex domain of finite type $\Om$ in $\C^2$. The estimate is not sharp as in the work of \cite{F-K88} and \cite{C-N-S92}, but Range's proof is much simpler and geometric; it is based on the earlier work of Catlin \cite{Cat89}, which shows that one can locally bump out the domain near each boundary point to a larger pseudoconvex domain. To construct the Leray maps $h_i$, Range uses Skoda's $L^2$ division theorem and obtained the following weighted $L^2$ estimate on the pushed out domain $D_p$ (for fixed $p \in bD$):    
\[
  \int_{D_p} \frac{|h^\eta_i(z,p)|^2}{|z-p|^2} \dist(z,D_p)^{2 \eta} \, dV(z) < C_{D,\eta}, \quad \eta >0. 
\]  
The domain $D_p$ is pseudoconvex and touches $b D$ only at $p$. 
In Range's proof, the functions $h_i(z,p)$ is defined only for $p \in b D$.  
For our proof of \rt{Thm::main}, we construct the functions $h_i(z,q)$ which are defined for $q$ in a thin shell outside the domain, i.e. for $q \in \Uc \sm D$ for some small neighborhood $\Uc$ of $\ov D$. 
More specifically, we show that for each $q \in \Uc \sm D$, there is a pseudoconvex domain $D_\ast(q)$ with $D \Subset D_\ast(q)$, $q \notin D_\ast(q)$, and whose boundary is as far away from $bD$ as possible, as measured by certain non-isotropic polydisks with centers in $D$. The polydisks we use are from \cite{Cat89} with some slight modification. Roughly speaking, the size of these polydisks are determined by the distance of its center and that of $q$ to $bD$. Once we have the $h_i(z,\zeta)$ which are holomorphic in $z \in D_\ast(q)$, we can apply Cauchy integral formula to estimate all the $z$ derivatives of the functions $h_i(z,q)$ for $z \in D$ and $q \in \Uc \sm D$. 

There is an additional problem of how $h_i(\cdot,q)$ depends on $q$. This problem arises since we need to integrate $h$ in the second variable in our solution operator. Here as in Range, we replace $h(z,q)$ with a smooth function $h(z,\zeta)$, by restricting $z$ in an interior domain $D'_\e \Subset D$, $D'_\e$ are all relatively compact in $D$ and approximate $D$ from inside as $\ve \to 0$. We can then  modify the function $h_i(z,\zeta)$ to a smooth function in $\zeta$ by restricting $\zeta$ to a small ball centered at $q$ and whose radius shrinks to $0$ as $\e$. In our case we need to control this radius, which can be estimated using the sup norm of $h_i(\cdot,q)$ on $\ov D_\e$. We show that the radius is small compared to $r(q) + \e$. This allows us to estimate the error arising from the switching of $q$ to $\zeta$, and in the end we obtain a modified $h_i(z,\zeta)$ which is holomorphic in $z \in D'_\e$ and smooth in $\zeta$ in $\Uc \sm D$. 

We can now solve the $\db$ equation on each approximating domain $D'_\e$, using the homotopy operator constructed in \cite{Gong19} and \cite{S-Y25}. The resulting estimate does not depend on $\e$ as $\e \to 0$, so we obtain a solution on the original domain $D$ by taking limits in suitable function spaces. We mention two auxiliary results which play crucial roles in our proof: the Hardy-Littlewood lemma for Sobolev and H\"older-Zygmund spaces (\rp{Prop::H-L}) which reduces the problem to estimating a weighted integral norm; and the commutator estimate (\rp{Prop::comm_est}) proved in \cite{S-Y24_1}, which allows us to harness the decay properties of the commutator term in the integral operator and essentially dispense with the need for integration by parts. 

As in Range \cite{Ran90}, the arbitrary small loss of regularity comes from the integrability condition in Skoda's theorem, and it seems that the current method is inadequate to remove this loss. 
The following question remains open. 
\begin{ques}
   Let $D \subset \C^2$ be a smoothly bounded pseudoconvex domain of finite type $m$. Does there exist an integral solution operator to the $\db$ equation that is bounded from $H^{s,p}(D)$ to $H^{s+\frac{1}{m},p}(D)$ for any $1<p<\infty$ and $s> \frac{1}{p}$, and $\La^s(D) $ to $\La^{s+\frac{1}{m}}(\ov D)$ for any $s>0$?  
\end{ques}
We expect that the method in this paper can be adapted to other pseudoconvex domains of finite type in $\C^n$. This will be left for future investigation.

We now fix some notations used in the paper. We write $x \lesssim y$ to mean that $x \leq Cy$ where $C$ is a constant independent of $x,y$, and we write $x \approx y$ if $x \lesssim y$ and $y \lesssim x$. We denote by $C^\infty_c(D)$ the space of $C^\infty$ functions with compact support in $D$. $\Oc(D)$ is the space of holomorphic functions on $D$. We use $D^l$ to denote a differential operator of order $l$: $D^l_z g(z)= \pa_{z_i}^{\all_i} \pa_{\ov z_j}^{\beta_j} g(z)$, $\sum_i \all_i + \sum_j \beta_j = l$. The volume element in $\C^n$ or $\R^N$ is denoted as $dV$.     
\vspace{5pt}     
\section{Function spaces}  
In this section, we recall some basic results for the Sobolev space $H^{s,p}(\Om)$ and the H\"older-Zygmund space $\La^s(\Om)$. All these results follow from standard theory, but their proofs are not readily found in the literature. We provide all the details for the reader's convenience. 

\begin{defn}[H\"older-Zygmund space on $\R^N$]  \label{Def::H-Z} 
The H\"older-Zygmund space on $\R^N$, denoted by $\La^s(\R^N)$ for $s\in\R^+$ is defined as follows
\begin{itemize}
    \item For $0<s<1$, $\La^s(\R^N)$ consists of all $f\in C^0(\R^N)$ such that 
    \[ 
    \|f\|_{\La^s(U)}:=\sup \limits_{\R^N} |f|+\sup\limits_{x,y\in \R^N, \, x \neq y} \frac{|f(x)-f(y)|}{|x-y|^s}<\infty. 
    \] 
    \item $\La^1(\R^N)$ consists of all $f\in C^0(\R^N)$ such that 
    \[ 
    \|f\|_{\La^1(\R^N)}:=\sup\limits_{\R^N} |f|+\sup\limits_{x,y\in \R^N, \, x \neq y}\frac{|f(x)+f(y)-2f(\frac{x+y}2)|}{|x-y|}<\infty. 
    \]
    \item For $s>1$ recursively, $\La^s(\R^N)$ consists of all $f\in \La^{s-1}(\R^N)$ such that $\nabla f\in\La^{s-1}(\R^N)$. We define $\|f\|_{\La^s(\R^N)}:=\|f\|_{\La^{s-1}(\R^N)}+\sum_{j=1}^d \|D_j f\|_{\La^{s-1}(\R^N)}$.
    \item We define $C^\infty(\R^N):=\bigcap_{s>0} \La^s(\R^N)$ to be the space of bounded smooth functions.
\end{itemize}
\end{defn} 
\begin{defn}[H\"older-Zygmund space on domains]   
Let $\Om \subset \R^N$ be a bounded Lipschitz domain. The H\"older-Zygmund space on $\Om$, denoted by $\La^s(\Om)$ for $s>0$,  is defined as 
$ \La^s (\Om) = \{ f: \exists \: \wti f \in \La^s (\R^N) \;  s.t. \; \wti f|_\Om = f \}$ equipped with the norm: 
\[
  |f|_{\La^s(U)}:= \inf_{\wti f\in \La^s (\R^N), \: \wti f|_\Om =f} | \wti f|_{\La^s (\R^N)}. 
\] 
\end{defn} 
\begin{rem}
  There is an intrinsic equivalent definition for the space $\La^s(\Om)$, namely, one which requires only that $f$ is defined in $\Om$, rather than assuming $f$ is the restriction of a function defined on the whole space. We will not use this definition in this paper. The interested reader can refer to \cite[Def. 1.120 and Thm 1.122]{Tri06} or \cite[Section 5]{Gong25}.     
\end{rem}                                                                       
Next we turn to the Sobolev spaces. 
We denote by $\Ss(\R^N)$ the space of Schwartz functions, and by $\Ss'(\R^N)$ the space of tempered distributions. For $g \in \Ss(\R^N)$, we set the Fourier transform $\hht g(\xi)=\int_{\R^N} g(x)e^{-2\pi i x \cdot \xi}dx$, and the definition extends naturally to tempered distributions. 
	
\begin{defn}
		We let $\dot\Ss(\R^N)$ denote the space\footnote{In some literature like \cite[Section 5.1.2]{Tri83}, the notation is $Z(\R^N)$.} of all infinite order moment vanishing Schwartz functions. That is, all $f\in\Ss(\R^N)$ such that $\int x^\alpha f(x)dx=0$ for all $\alpha\in\N^N$, or equivalently, all $f\in\Ss(\R^N)$ such that $\widehat f(\xi)=O(|\xi|^\infty)$ as $\xi\to0$. 
\end{defn}
	\begin{defn}[Sobolev space on $\R^N$] \label{Def::Sobolev_RN}
	Let $s\in\R$, $1<p<\infty$. We define $H^{s,p}(\R^N)$ to be the fractional Sobolev space consisting of all (complex-valued) tempered distribution $f\in\Ss'(\R^N)$ such that $(I-\Delta)^\frac s2f\in L^p(\R^N)$, and equipped with norm 
		\[ 
		\|f\|_{H^{s,p}(\R^N)} := \|(I-\Delta)^\frac s2f\|_{L^p(\R^N)}.
		\]
		Here $(I-\Delta)^\frac s2 $ is the Bessel potential operator given by 
		\begin{equation}\label{BesselPotent}
		    (I - \Del)^\frac s2 f=((1 + 4\pi^2|\xi|^2)^\frac s2 \hht f(\xi))^\vee.
		\end{equation} 	
\end{defn}
As in the case for the H\"older-Zygmund space, we define Sobolev space on domains as restrictions of the Sobolev space on $\R^N$. 
\begin{defn}[Sobolev space on domains] \label{Def::Sobolev_Dom} 
		Let $\Om \subset \R^N$ be an open set.
		\begin{enumerate}[(i)]
    		    \item Define $\Ss' (\Om): = \{\tilde f|_{\Om}:\tilde f\in \Ss' (\R^N) \}$.
    		    \item For $s \in \R$ and $1 < p < \infty$, define $H^{s,p}(\Om): = \{\tilde f|_{\Om}:\tilde f\in H^{s,p} (\R^N)\}$ with norm
		\[
		\| f \|_{H^{s,p}(\Om)} := \inf_{\wti{f}|_{\Om} = f} \|\tilde f\|_{H^{s,p} (\R^N)}.  
		\] 
		\item For $s \in \R$ and $1 < p < \infty$, define $H^{s,p}_0 (\Om)$ to be the subspace of $H^{s,p} (\R^N) $ which is the completion of $C^{\infty}_0 (\Om)$ under the norm $\| \cdot \|_{H^{s,p} (\R^N)}$. 
		    
		\end{enumerate}
\end{defn} 
\begin{rem}
    When $s=k$ is a non-negative integer, the space $H^{s,p}(\R^N)$ is the same as $W^{k,p}(\R^n)$, which consists of complex-valued functions whose derivatives up to order $k$ is in $L^p(\R^n)$.  
\end{rem}
For computation it is often convenient to use Littlewood-Paley characterizations of the above spaces. This leads to the Triebel-Lizorkin space and the Besov space, which generalize Sobolev space and H\"older-Zygmund space, respectively.  
decomposition.
\begin{defn}\label{Defn::Prem::DyaRes}
A \textit{classical dyadic resolution} is a sequence $\lambda=(\lambda_j)_{j=0}^\infty$ of Schwartz functions on $\R^n$, denoted by $\lambda\in\Cf$, such that the Fourier transforms $\widehat\lambda_j(\xi)=\int_{\R^n}\lambda_j(x)e^{-2\pi ix \cdot \xi} dx$ satisfies
\begin{itemize}
    \item $\widehat\lambda_0\in C_c^\infty\{|\xi|<2\}$, $\widehat\lambda_0|_{\{|\xi|<1\}}\equiv1$.
    \item $\widehat\lambda_j(\xi)=\widehat\lambda_0(2^{-j}\xi)-\widehat\lambda_0(2^{-(j-1)} \xi)$ for $j\ge1$ and $\xi\in\R^n$.
\end{itemize}
\end{defn} 
For any $f\in\Ss'(\R^n)$, we have the decomposition $f=\sum_{j=0}^\infty(\widehat\lambda_j\widehat f)^\vee=\sum_{j=0}^\infty\lambda_j\ast f$ and both sums converge as tempered distribution. We can define the Besov and Triebel-Lizorkin spaces using such $\lambda$.

A pair $(p,q)$ is said to be \textit{admissible} for the function class $\Bs_{pq}^s$ if $0<p,q\le\infty$, and admissible for the function class $\Fs_{pq}^s$ if $0<p<\infty$, $0<q\le\infty$ or $p=q=\infty$. 

\begin{defn}\label{Defn::Prem::BsFsDef}
Let $(\lambda_j)_{j=0}^\infty\in\Cf$, $s\in\R$ and let $(p,q)$ be admissible. We define the following norms: 
\begin{align*}
    \|f\|_{\Fs_{pq}^s(\lambda)}:=\|(2^{js}\lambda_j\ast f)_{j=0}^\infty\|_{L^p(\ell^q)}&=\left\|\big\|(2^{js}\lambda_j\ast f)_{j=0}^\infty\big\|_{\ell^q(\N)}\right\|_{L^p(\R^n)}.
\end{align*}
The \emph{Triebel-Lizorkin space $\Fs_{pq}^s(\R^n)$} is the set of all $f\in\Ss'(\R^n)$ such that $\|f\|_{\Fs_{pq}^s(\lambda)}<\infty$. One can show that the spaces do not depend on the choice of $\la \in \Cf$. Therefore we will henceforth use the norm $\|\cdot\|_{\Fs_{pq}^s(\R^n)}=\|\cdot\|_{\Fs_{pq}^s(\lambda)}$ for an (implicitly chosen) $\lambda\in\Cf$.

Let $\Omega\subseteq\R^n$ be an arbitrary open subset. We define $\Fs_{pq}^s(\Omega):=\{\tilde f|_\Omega:\tilde f\in\Fs_{pq}^s(\R^n)\}$ as subspaces of distributions in $\Omega$, with norms
\begin{equation} \label{F_space_ext_norm} 
   \|f\|_{\Fs_{pq}^s(\Omega)}:=\inf\{\|\tilde f\|_{\Fs_{pq}^s(\R^n)}:\tilde f|_\Omega=f\}.
\end{equation} 
\end{defn}
Notice that the above definition is not intrinsic to the domain since it requires the distribution to have an extension on the whole space. To define an intrinsic norm we need to impose certain conditions on the domain so it admits extension operators bounded on the function classes. We will use the Lipschitz boundary condition and Rychkov (universal) extension operator. 

We now define the notion of special and bounded Lipschitz domains. 
\begin{defn} \label{Def::Lip_dom}  
    A \emph{special Lipschitz domain} is an open set $\om \subset \R^N$ of the form $\omega=\{(x',x_N):x_N>\rho(x')\}$ with $\|\nabla\rho\|_{L^\infty}< L$. A \emph{bounded Lipschitz} domain is a bounded open set $\Om$ whose boundary is locally the graph of some Lipschitz function. In other words, $b \Om = \bigcup_{\nu=1}^M U_\nu$, where for each $1 \leq \nu \leq M $, there exists an invertible linear transformation $\Phi_\nu:\R^N \to \R^N$ and a special Lipschitz domain $\om_\nu$ such that 
    \[
      U_\nu \cap \Om = U_\nu \cap \Phi_\nu (\om_\nu). 
    \]
Fix any such covering $\{ U_\nu \}$, we define the \emph{Lipschitz norm of $\Om$ with respect to $U_\nu$, denoted as $\lip_{U_\nu}(\Om)$}, to be $\sup_{\nu} \| D \Phi_\nu\|_{C^0}$.  
\end{defn}

\begin{defn}\label{Defn::Space::ExtOp}
Let $\omega\subset\R^N$ be a \textit{special Lipschitz domain} of the form $\omega=\{(x',x_N): x_N >g(x')\}$ for some $g:\R^{N-1}\to\R$ such that $\|\nabla\rho\|_{L^\infty}<L$. 

The \textit{Rychkov's universal extension operator} $E=E_\omega$ for $\omega$ is given as follows: 
\begin{equation}\label{Eqn::Space::ExtOp}
\Ec_\omega f: =\sum_{j=0}^\infty\psi_j\ast(\1_{\omega}\cdot(\phi_j\ast f)),\qquad f\in\Ss'(\omega).
\end{equation}
Here $(\psi_j)_{j=0}^\infty$ and $(\phi_j)_{j=0}^\infty$ are families of Schwartz functions that satisfy the following properties: 
\begin{enumerate}[label=(\thesection.\arabic*)]\setcounter{enumi}{\value{equation}}
    \item\label{Item::Space::Scal} \textit{Scaling condition}: $\phi_j(x)=2^{(j-1)N}\phi_1(2^{j-1}x)$ and $\psi_j(x)=2^{(j-1)N}\psi_1(2^{j-1}x)$ for $j\ge2$.
	\item\label{Item::Space::Momt} \textit{Moment condition}: $\int\phi_0=\int\psi_0=1$, $\int x^\alpha\phi_0(x)dx=\int x^\alpha\psi_0(x)dx=0$ for all multi-indices $|\alpha|>0$, and $\int x^\alpha\phi_1(x)dx=\int x^\alpha\psi_1(x)dx=0$ for all $|\alpha|\ge0$.
	\item\label{Item::Space::Approx}\textit{Approximate identity}: $\sum_{j=0}^\infty\phi_j=\sum_{j=0}^\infty\psi_j\ast\phi_j=\delta_0$ is the Direc delta measure.
	\item\label{Item::Space::Supp} \textit{Support condition}: $\phi_j,\psi_j$ are all supported in the negative cone $-\Kb^L:=\{(x',x_N):x_N < -L |x'|\}$.\setcounter{equation}{\value{enumi}}
\end{enumerate}
We call $(\phi,\psi)$ a \emph{$\Kb^L$-Littlewood-Paley pair}. 
\end{defn} 
Note that condition \ref{Item::Space::Supp} implies that there exists some $c_0 >0$ such that $\supp \phi_0 \subset \{ x_N < - c_0' \}$. By the scaling condition \ref{Item::Space::Scal}, we have 
\begin{equation} \label{supp_phi_j} 
 \supp \phi_j \subset \{ x_N < - c_0 2^{-j} \}, \quad c_0 = 2 c_0'.      
\end{equation}

Since a bounded Lipschitz domain $\Om$ is locally a special Lipschitz domain, we can apply a partition of unity and patch together the extension operators for the special Lipschitz domains to obtain an Rychkov extension operator $\Ec_\Om$ for $\Om$. We omit the exact expression for $\Ec_\Om$ (see for example \cite[(3.3)]{S-Y25}) and we simply note that $\Ec_\Om$ have the same properties as $\Ec_\om$. 

Given a family $\phi=(\phi_j)_{j=0}^\infty$ satisfying properties \ref{Item::Space::Scal} - \ref{Item::Space::Supp}, we can define the following intrinsic norm: 
\[
  \| f \|_{\Fs^{s,in}_{pq}(\phi)} := 
  \|(2^{js}\lambda_j\ast f)_{j=0}^\infty\|_{L^p(\ell^q)}
  =\left\|\big\|(2^{js}\lambda_j\ast f)_{j=0}^\infty\big\|_{\ell^q(\N)}\right\|_{L^p(\R^n)}.  
\]
For a fixed $\phi$, we define $\Fs_{p,q}^{s,in}(\phi)$ as the space of distributions $f$ on $\om $ with finite norm $ \| f \|_{\Fs^{s,in}_{pq}(\phi)}$. By using partition of unity, we can similarly define the space $\Fs_{p,q}^{s,in}(\Om)$ for any bounded Lipschitz domain $\Om$. 

Rychkov in the paper \cite{Ryc99} proves the following remarkable extension theorem. 
\begin{prop}[Rychkov] \label{Prop::Rychkov} 
   Let $\om$ be a special Lipschitz domain. For any $s \in \R$, the operator $\Ec_\om$ is $\Fs^{s,in}_{pq} (\om) \to \Fs^s_{pq} (\R^N)$ bounded, provided that either $p<\infty$ or $p=q=\infty$. Using partition of unity, same properties hold for $\Ec_\Om$ for any bounded Lipschitz domain $\Om$. 
\end{prop} 
\begin{cor} \label{Cor::F_space} 
  Let $\Om$ be a bounded Lipschitz domain in $\R^N$ 
and let $(p,q)$ be a pair such that either $p <\infty$ or $p=q=\infty$. Then the intrinsic norm $\| \cdot \|_{\Fs^{s,in}_{pq} (\Om)}$ is equivalent to the extrinsic norm $\| \cdot \|_{\Fs^{s}_{pq} (\Om)}$ 
(see Definition \re{F_space_ext_norm}), and therefore $\Fs^{s,in}_{pq} (\Om) = \Fs^{s}_{pq} (\Om) $.   
\end{cor} 

Both the Sobolev space and the H\"older-Zygmund space are special cases of the Triebel-Lizorkin spaces, as shown by the following result. 
\begin{prop} \label{Prop::Equiv_spaces} 
   Let $\Om$ be either $\R^N$ or a bounded Lipschitz domain. Let $\phi = (\phi_j)_{j=0}^\infty$ be any Littlewood-Paley family satisfying properties \ref{Item::Space::Scal} - \ref{Item::Space::Supp}. 
\begin{enumerate}[(i)] 
    \item For all $1<p<\infty$ and $s \in \R$, $H^{s,p}(\Om)  = \Fs^s_{p2}(\Om) = \Fs^{s,in}_{p2}(\phi)$.    
    \item For all $s > 0$, $ \La^s(\Om) = \Fs^s_{\infty \infty}(\Om) = \Fs^{s,in}_{\infty \infty}(\phi) $. 
\end{enumerate}
\end{prop} 
\begin{proof}
   For the statement $H^{s,p}(\R^N) = \Fs^s_{p,2}(\R^N)$ for $s \in \R$ and $1 < p <\infty$, the proof is in \cite[p.~88]{Tri83}. For $\La^s(\R^N) = \Fs^s_{\infty \infty}(\R^N)$, $s>0$, the proof can be found in \cite[p.~90]{Tri83}. For bounded Lipschitz domain $\Om$, the spaces $H^{s,p}(\Om)$, $\La^s(\ov \Om)$ $\Fs^s_{pq}(\Om)$ are defined as the restrictions of the corresponding spaces on $\R^n$, so the statements follow from that of $\R^n$. The assertions $\Fs^s_{p2}(\Om) = \Fs^{s,in}_{p2}$ and $\Fs^s_{\infty \infty}(\Om) = \Fs^{s,in}_{\infty \infty}(\phi)$ follow immediately from \rc{Cor::F_space}. 
\end{proof}
As a consequence to \rp{Prop::Rychkov} and \rp{Prop::Equiv_spaces}, we have the following extension result for the Sobolev and H\"older-Zygmund space.   
\begin{cor} \label{Cor::Rychkov_ext} 
   Let $\Om$ be a bounded Lipschitz domain in $\R^N$. 
Then the Rychkov extension operator $\Ec_\Om$ defined on $\Om$ is bounded between the following spaces: 
\begin{enumerate}[(i)]
    \item $\Ec_\Om: H^{s,p}(\Om) \to H^{s,p}(\R^N)$, for any $s \in \R$ and $1 < p <\infty$. 
    \item $\Ec_\Om: \La^s(\ov \Om) \to \La^s(\R^N)$, for any $s >0$. 
\end{enumerate}
\end{cor}

\begin{prop}[Equivalence of norms] 
\label{Prop::Equiv_norm} 
Let $\Om$ be a bounded Lipschitz domain in $\R^N$. Let $0<p<\infty$, $0<q\le\infty$ or $p=q=\infty$. Then for any non-negative integer $k$, the following norms are equivalent. 
\[
  \| f \|_{\Fs^s_{pq}(\Om)} \approx \sum_{|\all| \leq k} \| D^\all f \|_{\Fs^{s-k}_{pq}(\Om)}.   
\]
\end{prop}
\begin{proof}
   This is \cite[Thm 1.1]{S-Y24_2}.    
\end{proof}
 
The following commutator estimate plays a crucial role in our proof. 
\begin{prop}{\cite[Proposition 5.8]{S-Y25}}   \label{Prop::comm_est} 
  Let $\Omega\subseteq\R^N$ be a bounded Lipschitz domain, and let $\Ec = \Ec_\Om$ be the Rychkov extension operator on $\Om$. Denote $\del_{b\Om}(x) := \dist(x, b\Om)$. Then for $1<p<\infty$ and $s >0$, the following estimates hold.  
\begin{gather*} 
  \bn \del_{b\Om}^{1-s}[D, \Ec]f  \bn_{L^p (\overline\Omega^c)}  \leq C_{s,p} \| f \|_{H^{s,p} (\Om)},\quad\forall f\in H^{s,p}(\Omega); 
  \\
  \bn \del_{b\Om}^{1-s}[D, \Ec]f \bn_{L^\infty (\overline\Omega^c)}  \leq C_s \| f \|_{\Lambda^s (\Om)},\quad\forall f\in \Lambda^s(\Omega). 
\end{gather*}  
\end{prop}
We collect a few more useful facts about the spaces:
\begin{defn}
	Let $X_0,X_1$ be two Banach spaces that belong to a larger ambient space. For $0<\theta<1$. The \emph{complex interpolation space $\blb X_0, X_1 \brb_\theta$} is defined to be the space consisting of all $f(\theta)\in X_0+X_1$, where $f:\{z\in\C:0\leq \RE z\leq 1\}\to X_0+X_1$ is a continuous map that is analytic in the interior, such that $f(it)\in X_0$ and $f(1+it)\in X_1$ for all $t\in\R$. The norm is given by
	\[  
	\|u\|_{[X_0,X_1]_\theta}=\inf_{f} \{\sup\limits_{t \in\R}(\|f(it)\|_{X_0}+\|f(1+it)\|_{X_1}):u=f(\theta)\}.
	\]  
\end{defn} 
\begin{prop} \label{Prop::interpol}   
  Let $\Om$ be a bounded Lipschitz domain in $\R^n$. Let $1<p<\infty$ and $s_0,s_1 \in \R$. Set $s_\thh = (1-\thh) s_0 + \thh s_1$, for $0< \thh < 1$. Then 
\begin{gather*}
\blb H^{s_0,p}(\Om), H^{s_1,p}(\Om) \brb_\thh 
=  H^{s_\thh,p}(\Om); 
\\
\blb \La^{s_0}(\Om), \La^{s_1}(\Om) \brb_\thh = \La^{s_\thh}(\Om).  
\end{gather*} 
\end{prop}
The reader can find the proof for the first statement in \cite[p.~70, Corollary 1.111]{Tri06}, and the proof of the second statement in \cite[p.~152, Theorem 6.4.5.]{B-L76}. Notice that in the reference $B^s_{pq}$ denotes the Besov space and we use the identification $\La^s (\Om)
= \Fs^s_{\infty,\infty}(\Om) = B^s_{\infty,\infty}(\Om)$. 

\begin{prop}[Complex interpolation theorem] \label{Prop::opt_interpol}   
	Let $X_0,X_1,Y_0,Y_1$ be Banach spaces that belong to some larger ambient spaces.  Suppose $T: X_0+X_1\to Y_0+Y_1$ is a linear operator such that for each $i=0,1$,  
	$\|Tu\|_{Y_i}\leq C_0\|u\|_{X_i}$ for all $u\in X_i$. Then $T:[X_0,X_1]_\theta \to [Y_0,Y_1]_\theta$ is bounded linear with $\|Tu\|_{[Y_0,Y_1]_\theta}\leq C_0^{1-\theta}C_1^\theta\|u\|_{[X_0,X_1]_\theta}$ for all $u\in[X_0,X_1]_\theta$.
\end{prop} 
\begin{proof}
   See \cite[Theorem 1.9.3(a)]{Tri95}  
\end{proof}

\begin{lemma} \label{Lem:add_supp}     
   Let $R_0 \in Z_+$ and let $\om = \{ x_N < g(x') \}$ be a special Lipschitz domain with $\| \na g \|_{L^\infty} \leq L- 2^{-R_0}$. Then for every $j \in \N$ and $a \in \R$, 
   \[
      \blb -\Kb^L \cap \{ x_N <- 2^{-j} \} \brb + \{ x_N - g(x') < a \}  \subset \{x \in \R^N: x_N - g(x') < a -
b L^{-1} 2^{-R_0-j} \}.  
   \]
\end{lemma}
\begin{proof}
  Let $u \in \{ x_N - g(x') < a \}$ and $v \in -\Kb^L \cap \{ x_N < - b 2^{-j} \}$, for $a \in \R$ and $b >0$. Then $u_N- g(u') < a$ and $v_N < - \max(L|v'|, b2^{-j})$. Using $\sup 
|\na \rho| \leq L-2^{-R_0}$ , we get 
\begin{align*}
 u_N + v_N  - g(u'+v') 
&\leq u_N - g(u') + v_N + |g(u') - g(u'+v')| 
\\ &\leq u_N - g(u') + v_N + \bl L - 2^{-R_0} \br |v'| 
\\ &\leq a + b L^{-1} 2^{-R_0} v_N 
\leq a - b L^{-1} 2^{-R_0-j}.   \qedhere 
\end{align*}
\end{proof}
\begin{lemma} \label{Lem::Liding}   
  Let $1 \leq p \leq \infty$ and $s>0$. Then for any $f  \in \Fs^t_{p,\infty}(\R^N)$ with $f \equiv 0$ on $\ov\Om^c$,  
\[
  \| f \|_{L^p(\Om,\del^{-t})} 
  \leq C_{\Om,t} \| f \|_{\Fs^t_{p,\infty}(\R^N)}.  
\]
The constant $C_{\Om,t}$ depends only on $t$ and the Lipschitz norm of $\Om$. 
\end{lemma}
\begin{proof}
The reader can find a proof in \cite[Prop.~5.3]{Yao24}. We include a slightly modified version here for the reader's convenience. We will prove for the case $\Om$ is a special Lipschitz domain and the general case follows by standard partition of unity argument. Let $\om = 
\{ x_N < g(x') \}$, where $x'= (x_1,\dots, x_{N-1})$. Assume that $|g|_{L^\infty(\R^{N-1})} \leq L$. Then 
\[
  \om + \left\{ x_N > L|x'| \right\} \subset \om. 
\]
We partition $\Om$ into dyadic strips.  
\[
  \om = \bigcup_{k \in \Z} S_k:= \bigcup_{k \in \Z} \left\{(x', x_N): 
-2^{\yh -k} < x_N - g(x') < -2^{-\yh -k} \right\}  
\]        
Since $\Kb_L \subset \om$, we can show that
\[
 \frac{1}{\sqrt{1+L^2}} | x_N - g(x')| 
 < \dist(x,b\om) < |x_N - g(x')|. 
\]
Denote $c_L:= (1+L^2)^{-\yh}$. Then  
\begin{equation} \label{Sk_bdy_dist_est} 
      c_L 2^{-\yh-k} < \dist(x,b\om) < 2^{\yh-k}, \quad \text{for $x \in S_k$}.  
\end{equation}

Let $f$ satisfy the hypothesis of lemma, and in particular $\supp f \in \om$. By \ref{Item::Space::Supp} and \re{supp_phi_j}, $\supp \phi_j \subset -\Kb^L \cap \{ x_N < - c_0 2^{-j} \}$. In view of \rl{Lem:add_supp}, there exists $R_0 >0$ such that
\[
  \supp \phi_j + \om \subset \blb -\Kb^L \cap \{ x_N <-c_0 2^{-j} \} \brb + \{ x_N - g(x') < 0 \} \subset \{ x \in \R^N: x_N < -c_0 L^{-1} 2^{-R_0-j} \}.   
\]
By \ref{Item::Space::Approx} we have $f = \sum_{j=0}^\infty \phi_j \ast f$. Observe that
\[
  \supp (\phi_j \ast f) = \supp \phi_j + \supp f \subset 2^{-j} + \om \subset \{ x \in \R^N: x_N < -c_0 L^{-1} 2^{-R_0-j} \}, 
\]
which is disjoint from $S_k$ if $-c_0 L^{-1} 2^{-R_0-j} < - 2^{\yh - k}$, or $j < k-R_0 - \yh + \log_2 (c_0 L^{-1})$. Let $R_1=\lceil R_0 - \yh + \log_2 (c_0 L^{-1}) \rceil$. Then $S_k \cap \supp (\phi_j \ast f) = \emptyset$ if $j< k -R_1$. Consequently we have $f(x) = \sum_{j =k-R_1}^\infty \phi_j \ast f $ for $x \in S_k$. Using \re{Sk_bdy_dist_est}, we have
\begin{align*}
   \|f \|^p_{L^p(\om,\del^{-t})} &= \sum_{k=0}^\infty \| f \|^p_{L^p(\om,\del^{-t})} 
   \leq \sum_{k=0}^\infty c_L^{-tp} 2^{(k+\yh)tp} \| f \|^p_{L^p(S_k)} 
   \\ &\leq c_L^{-tp}  \sum_{k=0}^\infty 2^{(k+\yh)tp} \Bn \sum_{j=k-R_1}^\infty | \phi_j \ast f| \Bn ^p_{L^p(S_k)}  
   \\ &=  c_L^{-tp} \sum_{k=0}^\infty \Bn \sum_{j = k-R_1}^\infty 2^{(k+\yh - j)t } 2^{jt} |\phi_j \ast f| \Bn^p_{L^p(S_k)}. 
\end{align*}
Let $l = j-k $, and the above expression is bounded by 
\begin{align*}
  c_L^{-tp} \sum_{k=0}^\infty \Bn \sum_{l = -R_1}^\infty 2^{(1-l)t } 2^{jt} |\phi_j \ast f| \Bn^p_{L^p(S_k)} &\leq c_L^{-tp} \sum_{k=0}^\infty \sum_{l = -R_1}^\infty 2^{(1-l)t }\Bn \sup_{j \geq k-R_1} 2^{jt} |\phi_j \ast f| \Bn^p_{L^p(S_k)}
  \\ &\leq c_L^{-tp} C_{R_1} \sum_{k=0}^\infty \Bn \sup_{j \in \N} 2^{jt} |\phi_j \ast f| \Bn^p_{L^p(\om)} = C_{\om,t} \| f \|_{\Fs^t_{p\infty}(\phi)}. 
\end{align*}
Here we use the convention $\phi_j \ast f = 0$ for $j \leq -1$. Notice that the constant $C_{\om,t}$ blows up as $L, t \to \infty$. 
\end{proof}
Let $\mr \Fs_{pq}(\R^N)$ be the norm closure of $C^\infty_c(\R^N)$ in $\Fs^s_{pq}(\R^N)$ and let $\mr \Fs_{pq} (\ov \Om)$ be the subspace of $\mr \Fs_{pq}(\R^N)$ defined by $\mr \Fs_{pq} (\ov \Om) = \blc f \in \mr \Fs_{pq}(\R^N): f|_{\ov \Om^c} = 0 \brc$.  
\begin{prop} \label{Prop::duality} 
   Let $\Om$ be a bounded Lipschitz domain or $\R^N$. Then for $s \in \R$ and $1 \leq p \leq \infty$, $\Fs^{-s}_{p1}(\Om) = \blb \mr \Fs^s_{p'\infty} (\Om) \brb'$ . 
\end{prop}
\begin{proof}
    By \cite[Remark 1.5]{Tri20}, we have $\Fs^{-s}_{p1}(\R^N) = \mr \Fs^s_{p'\infty} (\R^N)'$ for all $s \in \R$, $1 \leq p \leq \infty$, and $p' = \frac{p}{p-1}$. 
    Using the same argument as in \cite[Theorem 4.3.2/1]{Tri95}, we have $\mr \Fs_{pq} (\ov \Om)$ is the norm closure of $C^\infty_c(\Om)$ in $\Fs_{pq} (\R^N)$.   
    It follows that
    \begin{align*}
      \blb \mr \Fs^s_{p'\infty} (\ov \Om) \brb ' 
    &= \Fs^{-s}_{p1}(\R^N) / \blc f: \left< f,\phi \right>= 0,\; \forall \: \phi \in \mr \Fs^s_{p'\infty}(\R^N), \phi |_{\ov \Om^c} 
= 0 \brc
\\ &= \Fs^{-s}_{p1}(\R^N) / \blc f: \left< f,\phi \right>= 0, \; \forall \: \phi \in C^\infty_c(\Om) \brc 
\\ &= \Fs^{-s}_{p1} (\R^N) / \{ f: f|_\Om = 0 \} 
= \Fs^{-s}_{p1} (\Om).   \qedhere
    \end{align*}
\end{proof}
\begin{prop} \label{Prop::H-L} 
  Let $1< p < \infty$, $k \in \N$  and $s \in \R$ with $s \leq k$. Let $\Om$ be a bounded Lipschitz domain and denote by $\del_{b\Om}(x)$ the distance from $x$ to the boundary $b \Om$.  
\begin{enumerate}[(i)] 
    \item
  There exists a constant $C$ such that for all $u \in W^{k,p}_{\loc} (\Om)$ with $\| \del_{b\Om}^{k-s } D^{k} u \|_{L^p(\Om)} < \infty$, the following inequality holds
\begin{equation} \label{H-L_Sob_est}  
   \| u \|_{H^{s,p} (\Om)} \leq C \sum_{|\beta| \leq k} \| \del_{b\Om}^{k-s}  D^{\beta} u \|_{L^p(\Om)}.   
\end{equation}		
  \item
 There exists a constant $C$ such that for all $u \in W^{k,p}_{\loc} (\Om)$ with $\| \del_{b\Om}^{k-s} D^{k} u \|_{L^p(\Om)} < \infty$, the following inequality holds
\begin{equation} \label{H-L_HZ_est} 
 \| u \|_{\La^s(\Om)} \leq C \sum_{|\beta| \leq k} \| \del_{b\Om}^{k-s}  D^{\beta} u \|_{L^\infty(\Om)}.     
\end{equation}
\end{enumerate}
\end{prop}  
\begin{proof}
(i) See \cite[Prop.~A.2]{S-Y24_1}.   
\\[5pt]  
(ii) By \rp{Prop::Equiv_norm} we have the following equivalence of norms: 
\[
  |u|_{\Om,s} \approx \sum_{|\all|\leq k} |D^\all u|_{\Om,s-k}, \quad \forall \; s \in \R, \; k \in \N.  
\]
It is clear that \re{H-L_HZ_est} is an immediate consequence of the following statement: for all $u \in L^\infty_\loc(\Om)$, 
\[
  |u|_{\Om,-t} \leq C \| \del_{b\Om}^t u \|_{L^\infty(\Om)}, \quad t>0. 
\]
By \rp{Prop::duality}, we have $\La^{-t}(\Om) = \Fs^{-t}_{\infty,\infty}(\Om) = \blb \mr \Fs^t_{1,1} (\Om) \brb'$, where $\mr \Fs^t_{1,1} (\Om)$ denotes the norm closure of $C^\infty_c(\Om)$ in $\Fs^t_{1,1}(\R^N)$.   
By definition we see that  
\begin{align*}
   \| f \|_{\Fs^t_{1,\infty}(\Om)} 
  &= \left\| \sup_{j \in \N} 2^{jt} |\la_j \ast f| \right\|_{L^1(\Om)} 
 \leq \int_{\Om} \sum_{j \in \N} 2^{jt} |\la_j \ast f(x)| \, dx = \sum_{j \in \N} 2^{jt} \| \la_j \ast f \|_{L^1(\Om)} = \| f \|_{\Fs^t_{1,1}(\Om)}.     
\end{align*}
Thus $\Fs^t_{1,1}(\Om) \subset \Fs^t_{1,\infty}(\Om)$ and $\mr \Fs^t_{1,1}(\Om) \subset 
\mr \Fs^t_{1,\infty}(\Om)$. 
It is clear that $\mr \Fs^t_{1,\infty}(\Om) \subset 
\{ f \in \Fs^t_{1,\infty}(\R^N): f \equiv 0 \: \text{on} \: \ov \Om^c\} $. 
Now, by \rl{Lem::Liding}, we have 
\[
  \{ f \in \Fs^t_{1,\infty}(\R^N): f \equiv 0 \: \text{on} \: \ov \Om^c\} \subset L^1 \bl \Om, \del^{-t} \br, \quad t>0.  
  \]
In summary we have $\mr \Fs^t_{1,1}(\Om) \subset L^1(\Om,\del_{b\Om}^{-t})$, and $|f|_{L^1(\Om,\del_{b\Om}^{-t})} \leq C |f|_{\Fs^t_{1,1}(\Om)}$ for any $f \in \mr \Fs^t_{1,1}(\Om)$. It follows that 
\[
 L^\infty (\Om, \del_{b\Om}^t) = \blb L^1(\Om,\del_{b\Om}^{-t}) \brb'\subset  \blb \mr \Fs^t_{1,1} (\Om) \brb' = \La^{-t}(\Om), 
\]
and $|u|_{\Om,-t} \leq C |u|_{L^\infty 
(\Om,\del_{b\Om}^t)}$ for all $u$. The proof is complete.  
\end{proof}
\vspace{5pt}     
\section{Construction of the holomorphic support function}

Let $D$ be a bounded pseudoconvex domain with $C^\infty$ boundary in $\C^2$ of finite type $m$. Let $r$ be the defining function of $D$ such that $D = \{ z \in U: r(z)<0 \}$, where $U$ is a neighborhood of the closure of $D$. 
Fix $p_0 \in bD$ and a small neighborhood $U_0$ of $p_0$. 
By definition of finite type (\cite{Koh72, B-G77}), $m$ is the least integer with the following property: if $h$ is any germ about the origin of a holomorphic map from $\C$ to $\C^2$ with $h(0) = p_0$ and $d h(0) \neq 0$, then $r \circ h$ vanishes to order at most $m$ at the origin. 

For each $p \in U_0$, we introduce a special holomorphic coordinate system $\phi_p: \C^2 \to \C^2$ as in Catlin \cite{Cat89}. The map $\phi_p$ is constructed as follows. Assume $\pp{r}{z_2}(p) \neq 0$. We define $\Phi^1: \zeta^{(1)} \mapsto z$ by 
\begin{gather*}
   \zeta_1 = z_1 - p_1
\\ 
   \zeta_2 = 2 \pp{r}{z_1}(p) (z_1-p_1) + 2 \pp{r}{z_2}(p) (z_2 - p_2) . 
\end{gather*}
Denote $\rho^{(1)}(\zeta) = r \circ \Phi^{(1)} (\zeta)$. Then $\rho^{(1)}(0) = r(p)$, and 
\[
  \rho^{(1)}(\zeta) = r(p) + \RE \zeta_2 + O(|\zeta|^2). 
\]
In general, suppose that for $2 \leq l < m$, we have $\rho^{(l-1)} = r \circ \Phi^{(1)} \circ \cdots \Phi^{(l-1)} $, where 
 \[
   \rho^{(l-1)}(\zeta) = r(p) + \RE w_2 + \sum_{\substack{j+k \leq l-1 \\ j,k \geq 1} } a_{j,k}(p) w_1^j \ov w_1^k  
   + \RE (2 b_l(p) w_1^l) + O(|w_1|^{l} + |w_2| |w|).   
 \]
 We then define $\Phi^l: \zeta \mapsto w$ by 
\begin{gather*}
   \zeta_1  = w_1 , 
   \quad
   \zeta_2 = w_2 + 2 b_l(p) w_1^l . 
\end{gather*}
So then $\rho^{(l)} (\zeta) = r \circ \Phi^{(1)} \cdots \circ \Phi^{(l)} (\zeta)$ takes the form 
\[
  \rho^{(l+1)} (\zeta) = r(p) + \RE \zeta_2 
  + \sum_{\substack{j+k \leq l \\ j,k \geq 1} } a_{j,k}(p) \zeta_1^j \ov \zeta_1^k + O(|\zeta_1|^{l+1} + |\zeta_2| |\zeta|).   
\]
Let $\phi_p = \Phi^{(1)} \cdots 
\circ \Phi^{(m)}$. Then 
$\rho_p := r \circ \phi_p - r(p)$ takes the form 
\begin{equation} \label{rho_Taylor_exp}   
   \rho_p (\zeta) = \RE \zeta_2 + \sum_{\substack{j,k \geq 1 \\[1pt] j + k \leq m}  } a_{j,k} (p) \zeta_1^j \ov \zeta_1^k + O(|\zeta_1|^{m+1} + |\zeta_2 | |\zeta|).     
\end{equation}
$\rho_p$ is the defining function of the domain $\Om_p = \phi_p^{-1}(D)$.  
The coefficients $a_{j,k}(p)$ are linear combinations of products of derivatives of $r$ at $p$ up to order $m$ and depend smoothly on $p \in U_0$. 

For $l = 2, \dots, m$ and $\del >0$, set 
\[
 A_l(q) = \max \{ |a_{j,k} (q)|: j + k = l \}, 
\]
and 
\[
  \tau(q,\del) = \min \left\{ \left( \frac{\del}{A_l(q)} 
\right)^{\frac{1}{l}} : 2 \leq l \leq m \right\}.  
\]
By the assumption that $p_0 \in bD$ is of finite type $m$, we know that at least one of the functions $a_{j,k}(p_0)$ is non-zero (otherwise the curve $t \mapsto (\phi_{p_0} (t),0)$ has order of tangency with $bD$ greater than $m$ at $p_0$.) Hence by continuity and shrinking $U_0$, we can assume that the same function $a_{j,k}(q) \neq 0$ for all $q \in U_0$, and in particular $A_l(q) >0 $ for some $2 \leq l \leq m$.
It follows that 
\begin{equation} \label{tau_del_power_est} 
 \del^{\yh} \lesssim \tau (q,\del) \lesssim \del^{\frac{1}{m}}, \quad q \in U_0.     
\end{equation} 
The definition of $\tau(q,\del)$ also implies that if $\del'<\del''$, then 
\begin{equation} \label{tau_scaling}  
  (\del / \del'')^\yh \tau (q, \del'') 
  \leq \tau (q, \del') \leq  
  (\del / \del'')^{\frac{1}{m}} \tau (q',\del'').  
\end{equation}

For $\del \geq 0$ we define
\begin{equation} \label{J_del_def} 
   J_\del (p,\zeta) = \left[ \del^2 
+ |\zeta_2|^2 + \sum_{k=2}^m A_k(p)^2 |\zeta_1|^{2k} \right]^\yh.    
\end{equation}
For $a>0$ and $\gm \geq 0$, we define the generalized ``nonisotropic'' polydisk $P^{(a)}_{\del,\gm} (\zeta')$ centered at $\zeta'$ by 
\[
  P^{(a)}_{\del,\gm} (\zeta') = \left\{ \zeta \in \C^2: |\zeta_2 - \zeta'_2| < a (J_\del (p,\zeta') + \gm), \; |\zeta_1 - \zeta_1'| < \tau \bl p,a J_\del(p,\zeta') \br + a \gm \right\}.  
\]

\begin{lemma} \label{Lem::J_del}  
There exists a constant $a >0$ (independent of 
$p, \zeta', \del$), such for $\zeta' \in \Om_p \cap \{ |\zeta|<a \} $ and $\zeta \in P^{(a)}_{\del,\gm} (\zeta')$, the following estimates hold
\begin{gather*}
  J_\del(p,\zeta) \leq C_{m,r} \bl J_\del(p,\zeta') + a \gm \br;  
  \\ 
  J_\del(p,\zeta') \leq C_{m,r} \bl J_\del(p,\zeta) + a \gm \br.  
\end{gather*}  
\end{lemma} 
\begin{proof}
For simplicity we write $J_\del(\zeta') = J_\del(p,\zeta')$ and $\tau' = \tau(p,aJ_\del(p,\zeta')) $. 
  The condition $\zeta \in P^{(a)}_{\del,\gm}(\zeta')$ implies that 
\begin{equation} \label{zeta_diff} 
\left| |\zeta_2| - |\zeta'_2| \right| \leq a \bigl( J_\del(\zeta') + \gm \bigr), \quad  
\left| |\zeta_1| - |\zeta'_1| \right| \leq \tau' + a\gm.   
\end{equation}  
Thus for $\zeta \in P^{(a)}_{\del,\gm}(\zeta')$, we have
\begin{equation} \label{J_del_ub} 
\begin{aligned}
 J^2_\del(\zeta) 
&= \del^2 + |\zeta_2|^2 + \sum_{k=2}^m \blb A_k(p) \brb ^2 |\zeta_1|^{2k}  
\\ &\leq C_m \Blb \del^2 + |\zeta_2'|^2 + 
 \bl aJ_\del(\zeta') \br^2 + (a\gm)^2 + \sum_{k=2}^m 
\blb A_k(p) \brb^2\left( |\zeta_1'|^{2k} + (\tau')^{2k} + (a\gm)^{2k} \right) \Brb .      
\end{aligned}  
\end{equation}
By definition
\begin{equation} \label{tau'_est} 
  \tau' = \tau(p,a J_\del(\zeta')) \leq \left( \frac{a J_\del(\zeta')}{A_k(p)} \right)^{\frac{1}{k}}
 \quad \text{for} \; 2 \leq k \leq m.   
\end{equation} 
Plugging the above into \re{J_del_ub} and using the fact that $(a\gm)^{2k} < (a \gm)^2$ ($a$,$\gm$<1), we get
\begin{align*}
  J^2_\del(\zeta) &\leq C_{m,r} \Blb \del^2 + |\zeta_2'|^2 + \bl aJ_\del(\zeta') \br^2 + (a\gm)^2 + \sum_{k=2}^m \blb A_k(p) \brb ^2 |\zeta_1'|^{2k} \Brb 
  \\ &\leq C_{m,r} \bl \blb J_\del(\zeta') \brb^2 + (a\gm)^2 \br. 
\end{align*}

For the other direction, we have 
\begin{align*}
 \blb J_\del(\zeta') \brb^2  
&= \del^2 + |\zeta_2'|^2 + \sum_{k=2}^m \blb A_k(p) \brb^2 |\zeta_1'|^{2k}  
 \\ &\leq C_m \Blb \del^2 + |\zeta_2|^2 + 
 \bl aJ_\del(\zeta') \br^2 + (a\gm)^2 + \sum_{k=2}^m 
\blb A_k(p) \brb^2 \left( |\zeta_1|^{2k} + (\tau')^{2k} + (a\gm)^{2k} \right) \Brb  
\\ &= C_{m,r} \left( \blb J_\del(\zeta) \brb^2 +  \bl aJ_\del(\zeta') \br^2 + (a\gm)^2 \right) . 
\end{align*} 
Taking $a< \frac{1}{\sqrt{2 C_{m,r}}}$, we get $ \blb J_\del(\zeta') \brb^2 
\leq C'_{m,r} \Bl \blb J_\del(\zeta) \brb^2 + (a\gm)^2 \Br$. Taking square root on both sides we are done.
\end{proof}
\begin{prop} \label{Prop::polydisk}  
    There are positive constants $a_0$ and $c$, and for each $p \in U_0 \cap b D$, there is a family of pseudoconvex domains $\blc \Om_p^{\del} \brc_{0<\del<\del_0}$ with the following properties:
\begin{enumerate}[(i)]
    \item $0 \in b \Om_p^\ast$, where $\Om_p^\ast := int \blb \cap_{0<\del<\del_0} \Om_p^\del \brb$    
    \item $\{ \zeta \in \ov \Om_p: 0 < |\zeta| < c\} \subset \Om_p^\ast \subset \Om_p^{\del}$;
    \item For $\zeta' \in \Om_p$ with $|\zeta'| < c$ one has 
    \[
       P^{(a_0)}_{\del, \del} (\zeta') \subset \Om_p^{\del}.   
    \]
\end{enumerate}
The constant $a_0$ depends only on the defining function $r$ and it is independent of $\del$ and $\zeta'$.
\end{prop}
\begin{proof}
For simplicity we will write $J_\del(p,\zeta)$ as $J_\del(\zeta)$. Part (i) and (ii) are the same as in \cite[Prop.~2.4, (i),(ii)]{Ran90}, and we shall include the details here for the reader's convenience. 
  For small $s$ and $\del>0$, define 
\[ 
  W_{s,\del}(p) = \{ \zeta \in \C^2: |\zeta|<c \; \text{and} \; |\rho(\zeta)|<s J_\del(\zeta) \}.   
\] 
Let $H_{p,\del}$ be the smooth real-valued function on $W_{s,\del}(p)$ given by \cite[Prop.~4.1]{Cat89}. Catlin \cite[p.449-453]{Cat89} shows that $c,s, \ve_0$ and $\del_0$ can be chosen so that for all $0 < \ve < \ve_0$, $0 < \del \leq \del_0$, the set 
\[
  S_\del = \{ \zeta \in W_{s,\del}(p): \rho^\ve_\del(\zeta) = 0 \}, \quad \rho^\ve_\del(\zeta) := \rho(\zeta) + \ve H_{p,\del}(\zeta) 
\]
is a smooth pseudoconvex hypersurface (from the side $\rho^\ve_\del <0$). All the constants can be chosen independently of $p \in U_0 \cap bD$, and $s,\ve$ can be chosen independently of $\del$. Furthermore, $H_{p,\del}(\zeta) \approx - J_\del(\zeta)$. We now fix $s = s_0, \ve= \ve_0 >0$ independent of $\del$. It follows that  
\begin{equation} \label{Om_p_del} 
  \Om_p^\del = \{ |\zeta|<c: \rho(\zeta)<0 \} \cup 
  \{ \zeta \in W_{s_0,\del} (p): \rho^{\ve_0}_\del(\zeta) <0 \} 
\end{equation} 
is a pseudoconvex domain which satisfies 
\[
  \ov \Om_p \cap \{ |\zeta| < c \} \subset \Om_p^\del \subset \bl \Om_p \cap \{ |\zeta|<c \} \br \cup W_{s_0,\del} (p). 
\]
Define 
\begin{equation} \label{Om_p_ast}  
  \Om_p^\ast = int \bigcap_{0<\del<\del_0} \Om_p^\del. 
\end{equation} 
Then $\Om_p^\ast$ is pseudoconvex. Since  $\rho^{\ve_0}_\del(0) = \rho(0) +  \ve_0 H_{p,\del}(0) \approx -J_{\del}(p,0) = -\del$, it follows that $\dist(0,S_\del) \lesssim \del$ and thus $0 \in b\Om_p^\ast$, which proves (i).   

Next we prove (ii). Suppose $\zeta \in \ov{\Om_p}$ and $0<|\zeta|<c$. Then $\rho(\zeta) \leq 0$. 
\[
  \rho_\del^{\ve_0} (\zeta) = \rho(\zeta) + \ve_0 H_{p,\del}(\zeta) \leq - \ve_0 cJ_\del(\zeta) \leq - c \ve_0 \Bl 
|\zeta_1|^2 + \sum_{k=2}^m [A_k(p)]^2 |\zeta_2|^{2k} \Br^\yh . 
\]
For fixed $\zeta \neq 0$, the expression on the right-hand side is a negative constant independent of $\del$. 
Also it is clear that $\zeta \in W_{s_0,\del}(p)$ since $\rho(\zeta) \leq 0 < c_\zeta \leq s J_\del(\zeta)$ uniformly for all $\del$. Hence $\zeta \in \Om_p^\ast$. This proves (ii). 

We now prove (iii). First, we show that for $\zeta' \in \Om_p \cap \{ |\zeta|<c \}$, we have 
\begin{equation} \label{rho_polydisk_diff_est} 
  |\rho(\zeta) - \rho(\zeta')| < C_\rho a_0^{1/m} J_{\del}(\zeta'), \quad \forall \; \zeta \in P^{(a_0)}_{\del,\del} (\zeta'),  
\end{equation}
where $a_0 >0$ is to be determined. 
Denote 
\[
  R_b(p) = \blc \zeta \in \C^2: |\zeta_1| < \tau(p,b), \; |\zeta_2| < b \brc,   
\]

We observe that $\zeta \in R_{ J_\del(\zeta)} (p)$. Indeed, we have $|\zeta_2| < J_\del(\zeta)$ and for some $ 2 \leq k_\ast \leq m$, 
\[
  \tau (p, J_\del(\zeta) ) = \left( \frac{ J_\del(\zeta)}{A_{k_\ast} (p)} \right)^{\frac{1}{k_\ast}} \geq \left( \frac{ A_{k_\ast}(p) |\zeta_1|^{k_\ast} }{A_{k_\ast}(p)} \right)^{\frac{1}{k_\ast}} = |\zeta_1|. 
\]

By \cite[Prop.~1.2]{Cat89}, we have  
\[
  |D_1^l \rho(\eta)| \lesssim b \bigl[ \tau (\zeta',b) \bigr]^{-l}, \quad \eta \in R_b(\zeta').  
\]
Apply the above estimate with $l=1$ and $ b= 
J_{\del}(\eta)$ to get 
\begin{equation} \label{rho_zeta1_der_est} 
   |D_1 \rho(\eta)| \lesssim 
  J_{\del} (\eta) \bigl[ \tau (p, J_{\del}(\eta) \bigr]^{-1}, 
  \quad \eta \in  \phi_p^{-1} (U_0).   
\end{equation}
In particular the above holds for $\eta \in P^{(a_0)}_{\del,\del}(\zeta')$. 
We now compare $\ov C J_\del(\eta)$ and $\ov C J_\del(\zeta')$ for $\eta \in P^{(a_0)}_{\del, \del}(\zeta')$. By \rl{Lem::J_del} we have 
\begin{gather*}
   J_{\del}(\eta) \leq C_{m,\rho} \bl J_{\del}(\zeta') + a_0 \del \br \leq C'_{m,\rho} J_{\del}(\zeta'); 
\\ 
   J_{\del}(\eta) \geq \bl C_{m,\rho}^{-1} J_{\del}(\zeta') - a_0 \del \br \geq c_{m,\rho} J_{\del}(\zeta'), 
\end{gather*}
where we chose $a_0 < (2C_{m,\rho})^{-1}$. Hence 
\[ 
 |D_1 \rho(\eta)| \leq C_{m,\rho}  
   J_{\del} (\eta) \bigl[ \tau (p, J_{\del}(\eta) \bigr]^{-1} 
  \leq C_{m,\rho} J_{\del} (\zeta') \bigl[ \tau (p, J_{\del}(\zeta') \bigr]^{-1} , \quad \eta \in P^{(a_0)}_{\del, \del}(\zeta'). 
\]
By the mean value theorem and \re{rho_zeta1_der_est}, we have for $\zeta \in P^{(a_0)}_{\del, \del}(\zeta')$,  
\begin{equation} \label{rho_diff}  
\begin{aligned} 
  |\rho (\zeta) - \rho(\zeta')| 
  &\leq C_\rho J_{\del}(\zeta') \tau (p, J_{\del}(\zeta'))^{-1} |\zeta_1 - \zeta_1'| + C_\rho |\zeta_2 
- \zeta_2'|  
\\&\leq C_\rho J_{\del}(\zeta') \tau (p, J_{\del}(\zeta'))^{-1} \blb \tau (p,a_0 J_{\del}(\zeta'))  + a_0 \del \brb + C_\rho a_0 \blb J_{\del}(\zeta') + \del \brb . 
\end{aligned} 
\end{equation}
By \re{tau_scaling}, we have $\tau(p, a_0 J_{\del}(\zeta')) \leq a_0^{1/m} \tau (p, J_{\del}(\zeta')) $. Also $J_{\del}(\zeta') \tau (p, J_{\del}(\zeta'))^{-1} \leq 
1$ by \re{tau_del_power_est} (assume that $J_{\del}(\zeta')<1$). Thus \re{rho_diff} implies $|\rho (\zeta) - \rho(\zeta')| 
\leq C_\rho \blb a_0^{1/m} J_{\del}(\zeta') +  a_0  \del \brb \leq C_\rho'' a_0^{1/m} J_{\del}(\zeta')$ for all $\zeta \in P^{(a_0)}_{\del, \del}(\zeta')$. This proves \re{rho_polydisk_diff_est}.  

To finish the proof we consider two different cases. 
\\[5pt] 
\textit{Case 1. $\zeta' \in \ov \Om_p \cap \{ |\zeta|< c \}  \sm W_{s_0,\del} (p)$.} 

By definition of $W_{s_0,\del}(p)$ we have $\rho(\zeta')<-s_0 J_{\del}(\zeta')$. Together with \re{rho_polydisk_diff_est} this implies that 
\begin{align*}
  \rho(\zeta) \leq \rho(\zeta') + |\rho(\zeta)- \rho(\zeta') | \leq -s_0 J_{\del}(\zeta') + C_\rho a_0^{\frac{1}{m}} J_{\del}(\zeta') <  0,  \quad \zeta \in P^{(a_0)}_{\del,\del}(\zeta'), 
\end{align*}
where we choose $a_0 < [s_0 /C_\rho]^m$. Hence $\zeta \in \Om_p \subset \Om_p^{\del}$.    
\\[5pt] 
\textit{Case 2. $\zeta' \in \ov \Om_p \cap \{ |\zeta|< c \}  \cap W_{s_0,\del} (p)$.} 
It suffices to consider $\zeta \in P^{(a_0)}_{\del,\del}(\zeta') \sm \ov \Om_p$, in which case $\rho(\zeta) \geq 0$. By \re{rho_polydisk_diff_est}, we have  
\begin{align*}
  0 \leq \rho(\zeta) \leq \rho(\zeta') + |\rho(\zeta) - \rho(\zeta') | 
  &\leq C_\rho a_0^{1/m} J_{\del}(\zeta') . 
\end{align*}
By our choice of $a_0 < [s_0/C_\rho]^m$, the last quantity is bounded by $s_0 J_\del(\zeta')$, which means $\zeta \in W_{s_0,\del}(p)$. 
On the other hand, \rl{Lem::J_del} shows that $\rho(\zeta) \leq C_\rho a_0^{1/m} J_{\del}(\zeta') < C'_\rho a_0^{1/m} J_{\del}(\zeta)$. 
Hence
\begin{align*} 
  \rho^{\ve_0}_{\del}(\zeta) := \rho(\zeta) + \ve_0 H_{p,\del}(\zeta) 
\leq C_{\rho} a_0^{\frac{1}{m}} J_{\del}(\zeta) - \ve_0 c J_{\del} (\zeta), \quad \zeta \in P^{(a_0)}_{\del,\ti\del}(\zeta'). 
\end{align*}
The last expression is negative if we choose $a_0 < 
\blb \ve_0 c (C_\rho)^{-1} \brb^m$. This shows that $P^{(a_0)}_{\del, \del} (\zeta') \subset \Om_p^{\del}$.   
\end{proof}  
\begin{lemma} \label{Lem::del_zeta_equiv}   
  Let $\ti \zeta \in \{|\zeta|< c \} \sm \Om_p$ be such that $\dist(\ti \zeta, \Om_p) \approx |\ti \zeta|$. Suppose that there exists $\ti \del >0$ such that $\ti \zeta \in S_{\ti \del} = \{ \rho_{\ti \del}^\ve = 0 \}$, where $\ve < C(p,m)$. Then $\ti \del \approx |\ti \zeta|$.      
\end{lemma}
\begin{proof}
Since $\dist (\ti \zeta, \Om_p) = |\ti \zeta|$, we can assume that for some $\la_0,\la_1>0$, $\rho(\ti \zeta) \in (\la_0|\ti \zeta|, \la_1 |\ti \zeta|)$. 
  Since $\ti \zeta \in S_{\ti \del}$, we have $\rho(\ti \zeta) + \ve H_{p,\ti \del}(\ti \zeta) = 0$. By \cite[Prop.~4.1.]{Cat89} there exist constants $c_0, C_0 >0$ such that $ c_0 J_{\ti \del}(p,\ti \zeta) < -H_{p,\ti \del}(\ti \zeta) < C_0 J_{\ti \del}(p,\ti \zeta)$. Hence $\rho(\ti \zeta) \in (c_0 \ve J_{\ti \del}(p,\ti \zeta), C_0 \ve J_{\ti \del}(p,\ti \zeta))$. This implies that
\[
  c_0 \ve J_{\ti \del}(p,\ti \zeta) < \la_1 |\ti \zeta|, \quad \la_0 |\ti \zeta| < C_0 \ve J_{\ti \del}(p,\ti \zeta). 
\]
Since $\ti \del \leq J_{\ti \del} (p,\ti \zeta)$, the first inequality above shows that $\ti \del \leq (c_0\ve)^{-1} \la_1 |\ti \zeta| $. On the other hand, by the definition of $J_\del$, there exists $C_{\rho,p}$ such that 
$J_{\ti \del}(p,\ti \zeta) \leq \ti \del + C_{\rho,p} |\ti \zeta|$, so the second inequality shows that $\la_0 |\ti \zeta| \leq C_0 \ve (\ti \del + C_{\rho,p} |\ti \zeta|)$. If $\ve \leq \frac{\la_0}{2C_0 C_{\rho,p}}$, then $|\ti \zeta| \leq C_{\rho,p}^{-1} \ti \del$.   
\end{proof}

\begin{lemma} \label{Lem::zeta_j_seq} 
  There exist a sequence of points $\{ \zeta_j \} \to 0$, and constant $c_0,C_0$ such that $\dist(\zeta_j, b\Om_p) \approx |\zeta_j|$, $c_0 2^{-j} \leq  |\rho(\zeta_j)| \leq C_0 2^{-(j-1)}$, and $\rho(\zeta_j) + \ve H_{p,\del_j}(\zeta_j) = 0$, with $\del_j = 2^{-j}$. The constant $c_0,C_0$ depend only on $r$(or $\rho$), $p$ and $\ve$.      
\end{lemma} 
\begin{proof}
Denote $\rho^\ve_\del(\zeta) = \rho(\zeta) + \ve H_{p,\del}(\zeta)$. 
Take $\zeta$ to be of the form $(x_1,y_1,x_2,y_2) = (0,0,x_2,0)$. Since $\na \rho(0) = \pp{\rho}{x_2} (0) = 1$, we have $\dist (\zeta, \Om_p) \approx x_2$ if $x_2$ sufficiently small. There exist some $\la_0,\la_1>0$ such that $\rho(\zeta) \in (\la_0 x_2, \la_1 x_2)$. By \cite[Prop.~4.1.]{Cat89}, there exist $0<c<C$ such that $-C J_\del(p,\zeta) < H_{p,\del}(\zeta) < -c J_\del(p,\zeta)$. We seek $\zeta'$ in the form $\zeta'= (0,0,x_2',0)$ such that $\rho(\zeta') + \ve H_{p,\del} (\zeta') >0$. Since 
\[
  H_{p,\del}(\zeta') > -CJ_\del(p,\zeta') 
  = -C(\del^2 + (x'_2)^2)^\yh \geq -C (\del + x'_2), 
\]  
it suffices to find $\zeta'$ with $\rho(\zeta') - \ve C (\del +x'_2 ) = x'_2 - \ve C (\del +x'_2) > 0$, or equivalently, 
\[
  x_2' > \frac{\ve C}{1-\ve C} \del. 
\]
On the other hand, we seek $\zeta''$ in the form $(0,0,x_2'',0)$ such that $\rho(\zeta'') + \ve H_{p,\del} (\zeta'') < 0$. Since 
\[ 
 H_{p,\del}(\zeta'') < -c J_\del(p,\zeta'') 
 = -c (\del^2 + (x_2'')^2)^\yh 
 < -c'(\del + x''_2), \quad c' = c / \sqrt{2},  
\]
it suffices to find $\zeta''$ with $\rho(\zeta'') - \ve c' (\del + x''_2 ) = x''_2 - \ve c' (\del +x''_2) < 0$, or equivalently, 
\[
  x_2'' < \frac{\ve c'}{1-\ve c'} \del.  
\]
To summarize, we choose $\zeta' = (0,0,x_2',0)$ and $\zeta''= (0,0,x_2'',0)$ such that 
\[
  |\zeta''| = x_2'' < \frac{\ve c'}{1-\ve c'} \del < \frac{\ve C}{1-\ve C} \del < x_2' = |\zeta'|. 
\]
For example, we can choose $x_2'' = \frac{1}{2} \frac{\ve c}{1-\ve c} \del$ and $x_2' = 2 \frac{\ve C}{1-\ve C} \del$.  
Since $\rho_\del^\ve(\zeta)$ is a continuous function in $\zeta$, there exists $a_\ast \in \bl \yh \frac{\ve c'}{1-\ve c'} \del, 2 \frac{\ve C}{1-\ve C} \del \br $ such that $\rho^\ve_\del \bl (0,0,a^\ast,0) \br = 0$. 
We now set $\del= \del_j= 2^{-j}$, and let $\zeta_j = (0,0,a_j^\ast,0)$ be such that $\rho^\ve_{\del_j} \bl (0,0,a_j^\ast,0) \br = 0$. In other words $\rho(\zeta_j) + \ve H_{p,\del_j}(\zeta_j) = 0$. 
Setting $c_0 = \yh \frac{\ve c'}{1-\ve c'} $ and $C_0 = \frac{\ve C}{1-\ve C} $, we then have $|\zeta_j| = a_j^\ast \in (c_0 2^{-j}, C_0 2^{-(j-1)})$. 
Since $\rho(\zeta_j) \approx \dist(\zeta_j,\Om_p) = |\zeta_j|$, by adjusting the constants $c_0,C_0$ we get $\rho(\zeta_j) \in (c_0 2^{-j}, C_0 2^{-(j-1)})$.   
\end{proof} 
\begin{prop} \label{Prop::r_lb}     
  Let $p \in U_0 \cap bD$ and let $\Om_p^\ast$ be the domain from \rp{Prop::polydisk}. There are positive constants $c'$ and $\gm$ which are independent of $p$ such that  
\[
  r(z) \geq c' |z-p|^m \quad \text{for} \quad |z-p| \leq \gm \quad \text{and} \quad z \notin \phi_p (\Om_p^\ast). 
\]
\end{prop} 
\begin{proof}
The proof is in Range \cite{Ran90} and we will provide the details here for the reader's convenience. 
It suffices to show that if $|\zeta_0|< \gm'$ for some $\gm'>0$ and $\zeta_0 \notin \Om_p^\ast$, then $\rho(\zeta_0) \geq c|\zeta|^m$. By definition of $\Om_p^\ast$, there exists $\del_0>0$ such that $\zeta_0 \notin \Om_p^{\del_0}$, where 
\[
 \Om_p^{\del_0} = \{ |\zeta|<c: \rho(\zeta)<0 \} \cup \{ \zeta \in W_{s,\del_0}(p): \rho^{\ve_0}_{\del_0}(\zeta)<0 \}.
\]
Hence $0 < \rho^{\ve_0}_{\del_0}(\zeta_0) := \rho(\zeta_0) + \ve_0 H_{p,\del_0} (\zeta_0)$, or $\rho(\zeta_0) > -\ve_0 H_{p,\del_0}(\zeta_0) \approx \ve_0 J_{\del_0}(p,\zeta_0)$. 
By \re{J_del_def}, we have $J_{\del_0}(p,\zeta_0) \gtrsim |\zeta_2| + |\zeta_1|^m \gtrsim |\zeta|^m$, and the proof is done.    
\end{proof}
Since a pseudoconvex domain $D \subset \C^2$ of finite type is regular in the sense of Diederich and Fornaess \cite{D-F77}, it follows that there is a pseudoconvex domain $\hht{D}$ with 
\[
  \ov D \subset \hht{D} \Subset \{ z: r(z) < c \gm^m \}. 
\] 
Let $\mu = \sup \{ r(z): z \in b\hht{D} \}$; then $0<\mu < c\gm^m$, and we can choose $\gm'>0$, such that $0<\mu< c (\gm')^m < c \gm^m$. 
Let $\Om_p^\del$ and $\Om_p^\ast$ be given by \re{Om_p_del} and \re{Om_p_ast}. 
For each $0<\del < \del_0$, we define the domain $D_\del$ by  
\begin{equation} \label{D_del_def} 
   D_{\del} := \blb \hht{D} \cap \{ z: |z-p|<\gm \} \cap \phi_{p} \bl \Om_{p}^{\del} \br \brb \cup 
  \{ \hht{D} \cap \{ z: |z - p | > \gm' \}.   
\end{equation}
We also define the domain $D_\ast$ by 
\begin{equation} \label{D_ast_def} 
  D_\ast := \blb \hht{D} \cap \{ z: |z-p|<\gm \} \cap \phi_{p} \bl \Om_{p}^\ast \br \brb \cup 
  \{ \hht{D} \cap \{ z: |z - p | > \gm' \}.    
\end{equation}
\begin{prop} \label{Prop::Dq_pscx}   
For each $0 < \del < \del_0$, the domain $D_\del$ is pseudoconvex. The same holds for $D_\ast$. 
\end{prop}
\begin{proof}
The proof is a slight modification of \cite[Lem.~3.15]{Ran90}. Denote
\[
  D_\del^1 := \hht{D} \cap \{ z: |z-p|<\gm \} \cap \phi_{p} \bl \Om_{p}^{\del} \br, \quad 
  D_\del^2 :=  \{ \hht{D} \cap \{ z: |z - p| > \gm' \}.  
\] 
We show that $D_\del$ is pseudoconvex at any boundary point $w \in b D_\del$. If $|w-p| < \gm'$ (resp. $|w-p| > \gm$), then $w \in bD^1_\del$ (resp. $bD^2_\del$). Assume now that $\gm' \leq |w-p| \leq \gm$. Suppose that $w \notin \phi_{p} (\Om_p^{\del})$, then in particular $w \notin \phi_p (\Om_{p}^\ast)$ and \rp{Prop::r_lb} shows that $r(w) \geq c|w - p|^m \geq c (\gm')^m$. On the other hand, we know that $w \in \ov{\hht{D}}$, so that $r(w) \leq \mu$, which is a contradiction as we had chosen $\mu < c(\gm')^m$. Hence we must have $w \in \phi_{p} (\Om_{p}^{\del})$. This forces $w \in b\hht{D}$, and thus $D_\del$ is pseudoconvex at $w$. The proof for $D_\ast$ is similar and we leave the details to the reader. 
\end{proof}
We denote by $P^{b_1,b_2}(\zeta')$ the polydisk \[
  P^{b_1,b_2}(\zeta') = \bigl\{ \zeta \in \C^2: 
|\zeta_1 - \zeta_1'| < b_1, |\zeta_2 - \zeta_2'| < b_2 \bigr\}.    
\]

\begin{prop} \label{Prop::h_der_Cauchy_est}  
  Let $h \in \Oc(\Om)$ and $\zeta' \in \Om$ be such that $P^{(b_1,b_2)}(\zeta') \subset \Om$. Then the following estimate holds
\begin{equation} \label{h_der_Cauchy_est} 
   |D^{\all} h(\zeta')| \leq \frac{C_\all} {b_1^{\all_1+1} b_2^{\all_2+1}} \| h \|_{L^2(P^{(b_1,b_2)}(\zeta'))} .     
\end{equation}
\end{prop} 
\begin{proof}
Let $0<r_1 < b_1$ and $0 < r_2 < b_2$. Denote $P^{(b_1,b_2)} := \{ \zeta \in \C^2: |\zeta_1| < r_1, |\zeta_2| < r_2 \}$.    
By the Cauchy integral formula, 
\[
 h(z) = \frac{1}{(2\pi i)^2} \int_{b P^{(b_1,b_2)}(z)} \frac{h(\zeta)}{(\zeta_1-z_1)(\zeta_2-z_2)} \, d\zeta_1 \, d \zeta_2.  
\]
Taking derivatives and writing in polar coordinates $\zeta_j =z_j + r_j e^{i \thh_j}$ with $\thh_j \in [0,2\pi]$, we get  
\begin{align*}
   D^\all h(z) = \frac{\all ! i^2}{(2 \pi i)^2} \int_{[0,2\pi]^2} \frac{h(\zeta)(r_1 e^{i\thh_1}) (r_2 e^{i \thh_2})}{(\zeta_1-z_1)^{\all_1+1} (\zeta_2-z_2)^{\all_2+2}  } \, d \thh_1 \, d \thh_2. 
\end{align*}
Multiply $r_1^{\all_1+1}r_2^{\all_2+1}$ on both sides: 
\[
 |D^{\all} h(z)| r_1^{\all_1+1} r_2^{\all_2+1} \leq \frac{\all !}{(2 \pi)^2} \int_{[0,2\pi]^2} |h(\zeta)| r_1 r_2 \, d \thh_1 \, d \thh_2 
\]
Now integrate $r_j$ from $0$ to $b_j$ and apply H\"older's inequality to get
\begin{align*}
   \frac{b_1^{\all_1+2}}{\all_1+2} \frac{b_2^{\all_1+2}}{\all_2+2} |D^{\all} h(z)| 
   &\leq \frac{\all !}{(2 \pi)^2} \int_{r_1=0}^{b_1} \int_{r_2=0}^{b_2} \int_{\thh_1=0}^{2 \pi}  \int_{\thh_1=0}^{2 \pi} |h(\zeta)| r_1 r_2 \, dr_1 \, dr_2 \, d \thh_1 \, d \thh_2 
\\ &\leq \frac{\all !}{(2 \pi)^2} \Bigl( \, 
\int\limits_{\substack{0 \leq r_j \leq b_j \\ 0 \leq 
\thh_j \leq 2 \pi}}  |h(\zeta)|^2 (r_1 r_2)^2 \, dr_1 \, dr_2 \, d \thh_1 \, d \thh_2 \Bigr)^\yh (b_1 b_2)^\yh (2 \pi)
\\ &\leq \frac{\all !}{2 \pi} b_1 b_2 
\left( \int_{P^{(b_1,b_2)}(\zeta')} |h(\zeta)|^2 \, dV(\zeta) \right)^\yh, 
\end{align*}
which proves \re{h_der_Cauchy_est}.  
\end{proof}  
Let $\ti \zeta \notin \Om_p^{\del}$ and $|\ti \zeta| < C \del$. Denote $d_{\del}(z) := \dist(z,b\Om_p^{\del})$. Suppose $h \in \Oc( \Om_p^{\del})$ satisfies 
\begin{equation} \label{M_eta_h}  
 \blb M^\eta_{\ti \zeta, \del} (h) \brb^2 := \int_{\Om^{\del}_p} \frac{|h(\zeta)|^2 (d_{\del})^{2\eta} (\zeta)}{|\zeta- \ti \zeta|^2}  \, dV(\zeta) < \infty.     
\end{equation} 

Let $\zeta' \in \Om_p \cap \{|\zeta| < c \}$, and let $c, a_0$ be the constants from \rp{Prop::polydisk}.  Denote 
\begin{equation} \label{beta_12}  
\begin{gathered} 
 \beta_2 (\zeta') := (a_0/2) \bl J_{\del}(p,\zeta') + \del \br \approx a_0 J_{\del}(p,\zeta'), 
 \\ 
  \beta_1(\zeta') := \tau \bl p, (a_0/2) J_{\del}(p,\zeta') \br + (a_0/2) \del 
  \approx 
  \tau \bl p, a_0 J_{\del}(p,\zeta') \br. 
\end{gathered} 
\end{equation} 

We now apply \rp{Prop::h_der_Cauchy_est} to $P^{(\beta_1,\beta_2)} (\zeta') = P^{(a_0/2)}_{\del,\del} (\zeta') \subset \Om_p^{\del}$, and we obtain the following: 
\begin{equation} \label{h_Cauchy_est}  
 |D^\all h(\zeta')| \leq C_\all \frac{1}{\beta_1^{\all_1+1}} \frac{1}{\beta_2^{\all_2 +1}} \left[\int_{P^{(\beta_1,\beta_2)}(\zeta')} |h(\zeta)|^2 \, dV(\zeta) \right]^{1/2} .  
\end{equation} 
We estimate 
\begin{equation} \label{h_L2_polydisk_parts}   
\begin{aligned}
\int_{P^{(\beta_1,\beta_2)}(\zeta')} |h(\zeta)|^2 \, dV(\zeta) 
&= \int_{P^{(\beta_1,\beta_2)}(\zeta')} \frac{|h(\zeta)|^2 (d_{\del})^{2\eta} (\zeta)}{|\zeta- \ti \zeta|^2} \frac{|\zeta- \ti \zeta|^2}{(d_{\del})^{2 \eta} (\zeta)}  \, dV(\zeta) 
\\ &\quad < \blb M^\eta_{\ti \zeta, \del} (h) \brb^2
\sup_{\zeta \in P^{(\beta_1,\beta_2)}(\zeta')} \frac{|\zeta- \ti \zeta|^2}{(d_{\del})^{2 \eta} (\zeta)}. 
\end{aligned}
\end{equation}

To proceed, we need to estimate $|\zeta-\ti \zeta|$ from above and $d_{\del}(\zeta)$ from below, for $\zeta \in P^{(\beta_1,\beta_2)}(\zeta') $. 
We have 
\begin{equation} \label{zeta_zeta'_ub}  
\begin{aligned} 
  |\zeta- \ti \zeta| 
&\leq |\zeta-\zeta'| + |\zeta'| + |\ti \zeta| \\ &\leq \beta_1(\zeta') + \beta_2(\zeta') + |\zeta'| + C \del
\\ &\leq C_\rho \tau \bl p, J_{\del}(p,\zeta') \br + |\zeta'|,   
\end{aligned}  
\end{equation}
where we used the assumption $|\ti \zeta| < C \del$. (In fact, we only need $|\ti \zeta| < C \del^\yh$ since $\tau \bl p, J_{\del}(p,\zeta') \br \geq c [J_\del(p,\zeta')]^\yh \geq c\del^{\yh}$.)     
Note that that the constant $C_\rho$ depends on $a_0$. 
By \rp{Prop::polydisk}, we have 
\[
  P^{(a_0)}_{\del,\del} (\zeta') = \left\{ \zeta \in \C^2: |\zeta_2 - \zeta'_2| < a_0 (J_{\del} (p,\zeta') + \del), \; |\zeta_1 - \zeta_1'| < \tau \bl p,a_0 J_{\del}(p,\zeta') \br + a_0 \del \right\} \subset \Om_p^{\del}.  
\]
Thus for $w \in b\Om_p^{\del}$, we must either have 
\\ (i) 
$|w_2- \zeta'_2| \geq a_0 \bl J_{\del}(p,\zeta') + \del \br = 2 \beta_2$;    
\\ or \\ 
(ii) $|w_1-\zeta'_1| \geq \tau \bl p,a_0 J_{\del}(p,\zeta') \br + a_0 \del  $. 
\\[5pt] If $|w_2- \zeta'_2| \geq 2 \beta_2$, then  
\begin{equation} \label{w2-zeta2_lb}  
  |w_2 - \zeta_2| \geq |w_2 - \zeta_2'| - |\zeta_2' - \zeta_2| 
 \geq 2 \beta_2 - \beta_2 = \beta_2 \approx a_0 J_{\del}(p,\zeta').     
\end{equation}
If  $|w_1 - \zeta_1'| \geq \tau \bl p,a_0 J_{\del}(p,\zeta') \br + a_0 \del $, then by using \re{tau_scaling},  
\begin{equation} \label{w1-zeta1_lb} 
\begin{aligned}
    |w_1 - \zeta_1| 
  &\geq |w_1 - \zeta_1'| - |\zeta_1' - \zeta_1| \\ &\geq \tau \bl p,a_0 J_{\del}(p,\zeta') \br + a_0 \del - \blb \tau \bl p, \frac{a_0}{2} J_{\del}(p,\zeta') \br + \frac{a_0}{2} \del \brb 
  \\&\geq (2^{1/m}-1) \tau \Bl p, \frac{a_0}{2} J_{\del}(p,\zeta') \Br  + \frac{a_0}{2} \del 
   \\ &\geq C_\rho \blb \frac{a_0}{2} J_{\del}(p,\zeta') \brb^{1/2} \geq C'_\rho J_{\del}(p,\zeta'). 
\end{aligned}
\end{equation}
Together, \re{w2-zeta2_lb} and \re{w1-zeta1_lb} imply that $|w-\zeta| \geq C_\rho J_{\del}(p,\zeta')$, which means 
\begin{equation} \label{d_ti_del_lb} 
   d_{\del}(\zeta) \geq C_\rho J_{\del}(p,\zeta') \quad \text{for} \quad \zeta \in P^{(a_0/2)}_{\del,\del}(\zeta'). 
\end{equation} 
Using \re{zeta_zeta'_ub} and \re{d_ti_del_lb} in \re{h_L2_polydisk_parts}, we get 
\[
  \int_{P^{(a_0/2)}_{\del,\del}(\zeta')} |h(\zeta)|^2 \, dV(\zeta) \leq C_\rho \blb M^\eta_{\ti \zeta, \del} (h) \brb^2 \frac{\blb \tau(p,J_{\del}(p,\zeta') + |\zeta'| \brb^2}{ \blb J_{\del}(p,\zeta') \brb^{2 \eta} }.  
\]
Using the above estimate and \re{beta_12} in \re{h_Cauchy_est}, and denoting $\tau = \tau(p, J_{\del}(p,\zeta')) $, we get 
\begin{equation} \label{h_der_est_tau_beta}   
\begin{aligned}
  |D^\all h(\zeta')| 
 &\leq C_{\rho,\all} \frac{1}{\tau^{\all_1+1}} \frac{1}{\blb J_{\del}(p,\zeta') \brb^{\all_2+1}}  \frac{\tau + |\zeta'|}{\blb J_{\del}(p,\zeta') \brb^\eta } M^\eta_{\ti \zeta, \del} (h) 
 \\ &\leq C_{\rho,\all}   
  \left( \frac{1}{\tau^{\all_1} [J_{\del}(p,\zeta')]^{\all_2+1+\eta} } + \frac{|\zeta'|}{\tau^{\all_1+1} [J_{\del}(p,\zeta')]^{\all_2+1+\eta}} \right) M^\eta_{\ti \zeta, \del} (h) .  
\end{aligned} 
\end{equation}
By the Taylor expansion of $\rho$ (\re{rho_Taylor_exp}), we have
\begin{equation} \label{beta_tau_lb}  
J_{\del}(p,\zeta') \geq c_\rho \bl \del + |\rho(\zeta')| + |(\zeta')_2| + |\zeta'|^m \br,  
\end{equation} 
which then implies 
\[
 \tau = \tau \bl p, J_{\del}(p,\zeta') \br 
 \geq c_\rho \blb J_{\del}(p,\zeta') \brb^{1/2} 
\geq c_\rho J_{\del}(p,\zeta') 
\geq c'_\rho \bl \del + |\rho(\zeta')| + |(\zeta')_2| \br.      
\]
On the other hand, by definition we have for some $2 \leq l_0 \leq m$, 
\[
  \tau\bl p, J_{\del}(p,\zeta') \br 
  = \left( \frac{J_{\del}(p,\zeta') }{A_{l_0}(p)} \right)^{1/l_0} \geq \left( \frac{A_{l_0}(p)|(\zeta')_1|^{l_0}}{A_{l_0}(p)} \right)^{1/l_0}
  = |(\zeta')_1|. 
\]
Hence 
\begin{equation}
  \tau\bl p, a_0 J_{\del}(p,\zeta') \br  
\geq c_\rho \blb \del + |\rho(\zeta')| + (\zeta')_1 + (\zeta')_2 \brb 
\geq c_\rho \bl \del + |\rho(\zeta')| + |\zeta'| \br.    
\end{equation}

Combining \re{h_der_est_tau_beta} with 
\re{beta_tau_lb}, we obtain the following 
\begin{lemma}
 Let $\ti \zeta \notin \Om_p^\del$ and $|\ti \zeta| < C \del$. Suppose $h \in \Oc(\Om_p^{\del})$ and $M^\eta_{\ti \zeta, \del} (h) < \infty$ for some $\eta >0$, where $M^\eta_{\ti \zeta, \del} (h)$ is given by \re{M_eta_h}. Then for $\zeta' \in \Om_p = \phi_p^{-1}(D)$ with $|\zeta'|< c$ the following estimate holds: 
\begin{equation}
  |D^\all h(\zeta')| 
\leq C_{\rho,\all} \frac{M^\eta_{\ti \zeta,\del}(h)   }{ \bl \del + |\rho(\zeta')| + |\zeta'| \br^{\all_1} 
\bl \del + |\rho(\zeta')| + |\zeta'_2| + |\zeta'_1|^m \br^{\all_2+1+\eta}}. 
\end{equation}
\end{lemma}  
We now pull back the estimates onto the original domain via the biholomorphic map $\phi_p$. Recall that $D_\del$ given by \re{D_del_def}, and is locally $\phi_p(\Om_p)$. 
For $q \notin D_\del$, we denote 
\[
  I^{\eta}_{q,\del} [h] := \left[ \int_{D_\del} \frac{|h(z)|^2}{|z - q|^2} \dist_{b D_\del}^{2\eta} (z) dV(z) \right]^\yh, \quad L_i = d \phi_p \left( \pp{}{\zeta_i} \right), \quad i = 1,2.  
\]
The following result is an immediate consequence of the lemma above.   
\begin{lemma} \label{Lem::h_est_D_del} 
 Let $q \notin D_\del$ and suppose that $|q - p| < C \del$. Let $h \in \Oc(D_\del)$ with $I^\eta_{q, \del} [h] < \infty$ for some $\eta >0$. Then for $z \in D$ with $|z-p| < c$ the following estimate holds: 
\begin{equation*}  
   |L_1^{\all_1} L_2^{\all_2} h(z)| \leq C_{r,\all} \frac{I^\eta_{q, \del} (h) }{\blb \del + |r(z)| + |z-p| \brb^{\all_1} \blb \del + |r(z)| + |g(p,z)| + |z-p|^m \brb^{\all_2+1+\eta} },      
\end{equation*} 
where $g(p,z):= (\phi_p^{-1})^{(2)}(z)$.  
\end{lemma} 

The following result is known as the Skoda's division  theorem, which we shall use to construct holomorphic support functions. 
\begin{prop}{\cite[Thm 1]{Sko72}}  \label{Prop::Skoda}    
  Let $\Om$ be an pseudoconvex domain in $\C^n$, and $\psi$ be a plurisubharmonic function in $\Om$. Let $\Psi_1,\Psi_2, \dots, \Psi_m$ be a system of $m$ functions holomorphic in $\Om$. Let $\all>1$ and $\mu = \min \{n,m-1 \}$. Let $f$ be a holomorphic function in $\Om$ such that 
  \[
    \int_\Om |f|^2 |\Psi|^{-2\all \mu -2} e^{-\psi} \, d V(z) < \infty. 
  \]
  Then there exist $h_1,h_2,\dots, h_m$ holomorphic in $\Om$ such that $f = \sum_{i=1}^m \Psi_i h_i$ and 
\[
  \int_\Om |h|^2 |\Psi|^{-2 \all \mu} e^{-\psi} \, d V(z) 
  \leq \frac{\all}{\all-1} \int_\Om |f|^2 |\Psi|^{-2\all \mu -2} e^{-\psi} \, d V(z). 
\]
\end{prop} 
For $q \in U_0 \sm D$ with $|q-p|<Cr(q)$, we want to apply Skoda's theorem to obtain holomorphic functions on a domain $D_\ast(q)$ associated with $q$, and satisfying 
\[
  \int_{D_\ast (q)} \frac{|h_{\eta,i}(q,z)|^2}{|z-q|^2} \blb \dist(z, D_\ast(q)) \brb^{2 \eta} \, dV(z) < C,  
\]
for some constant $C$ independent of $q$. 

We now describe how to construct such domain $D_\ast(q)$. 
Fix $p \in bD$, the neighborhood $U_0$ of $p$, the coordinate map $\phi_p$, and $\Om_p= \phi_p^{-1}(D)$. Let $\{ \zeta_j  \} \notin \ov \Om_p$ be the sequence given by \rl{Lem::zeta_j_seq}, such that $\zeta_j \in b\Om_p^{\del_j}$ and $|\rho(\zeta_j)| \approx \del_j =2^{-j}$. Denote $q_j := \phi_{p} (\zeta_j)$, so then $|r(q_j)| \approx \del_j$, for all $j$. 

For each $q \in U_0 \sm \ov D$ with $|q-p| < C r(q)$, let $j_\ast = j_\ast(q) = \max \{j: r(q_j) \leq r(q) < r(q_{j-1}) \}$, so that $r(q) \geq r(q_{j_\ast})$. Since $r(q_j) \in \bl c_0 2^{-j}, C_0 2^{-(j-1)} \br$ for all $j$, we have 
\[
  r(q) \leq r(q_{j_\ast-1}) \leq C_0 2^{-(j_\ast-2)} 
  = 4 C_0 c_0^{-1} c_0 2^{-j_\ast} \leq C_1 r(q_{j_\ast}), \quad C_1 =  4 C_0 c_0^{-1}. 
\]
Hence $r(q) \approx r(q_{j_\ast})$, which implies that $|q-p| < C r(q) \approx C r(q_{j_\ast}) \approx C \del_{j_\ast}$.  

Let $D_\ast(q):= D_{\del_{j_\ast}}$, where $D_\del$ is defined by \re{D_del_def}. 
By \rp{Prop::Dq_pscx}, $D_\ast(q)$ is pseudoconvex, so we can apply \rp{Prop::Skoda} to $D_\ast(q)$.
\begin{prop} \label{Prop::Skoda_D_ast}    
For each $\eta>0$ there is a constant $C_{D,\eta}$ such that for $q \in U_0 \sm \ov{D}$  there are functions $h_{\eta,i} (q,\cdot) \in \Oc(D_\ast(q))$, $i=1,2$ such that
\begin{equation} \label{Dqjast_div} 
h_{\eta,1}(q,z) (z_1 - q_1) + h_{\eta,2}(q,z) (z_2 - q_2) = 1, \quad z \in D_\ast (q).  
\end{equation}
 Furthermore, we have 
\begin{equation} \label{Dqjast_h_est} 
  I^\eta_q(h) := \int_{D_\ast (q)} \frac{|h_{\eta,i}(q,z)|^2}{|z-q|^2} \blb \dist(z, D_\ast(q)) \brb^{2 \eta} \, dV(z) < C_{D,\eta}.
\end{equation} 
\end{prop}  
\begin{proof}
Since $D_\ast (q)$ is pseudoconvex, the function $-\log \dist(z, D_\ast (q))$ is a plurisubharmonic function in $D_\ast (q)$.     
We show that $q \notin D_\ast (q) = D_{\del_{j_\ast}}$. Denote $\zeta_0 = \phi_p^{-1}(q)$ and $\zeta_{j_\ast} = \phi_p^{-1}(q_{j_\ast})$. By the definition of $D_{\del_{j_\ast}}$, it suffices to show that $\rho_{\del_{j_\ast},\ve}(\zeta_0):= \rho(\zeta_0) + \ve H_{p,\del_{j_\ast}}(\zeta_0) \geq  0$. Note that since $\zeta_{j_\ast} \in b\Om_p^{\del_{j_\ast}}$, we have $\rho_{\del_{j_\ast},\ve} (\zeta_{j_\ast}) = 0$. 

Now, $J_{\del_{j_\ast}}(p,\zeta_0) \gtrsim J_{\del_{j_\ast}}(p,\zeta_{j_\ast})$, since $\zeta_{j_\ast} =(0,0,x^{(j_\ast)}_2,0)$, $\zeta_0 = (x_1,y_1,x_2,y_2)$, and $x_2 \approx \rho(\zeta_0) \geq \rho(\zeta_{j_\ast}) \approx x_2^{(j_\ast)} >0$.   
Using $- C_0 J(p,\zeta_{j_\ast}) \leq H_{p,\del_{j_\ast}} (\zeta_{j_\ast}) < -c_0 J_{\del_{j_\ast}}(p,\zeta_{j_\ast})$, we obtain 
\[
   H_{p,\del_{j_\ast}}(\zeta_0) \geq - C_0 J_{\del_{j_\ast}}(p,\zeta_0) 
   \geq -C_0'J_{\del_{j_\ast}}(p,\zeta_j) 
   = (C_0' c_0^{-1}) 
\blb -c_0 J_{\del_{j_\ast}}(p,\zeta_j) \brb \geq c_1 H_{p,\del_{j_\ast}}(\zeta_j).  
\]
It follows that  
\[
  \rho(\zeta_0) + \ve H_{p,\del_{j_\ast}} (\zeta_0)
 \geq c_1 \rho(\zeta_{j_\ast}) + \ve c_1 H_{p,\del_{j_\ast}}(\zeta_{j_\ast}) = c_1 \rho_{\del_{j_\ast},\ve} (\zeta_{j_\ast}) = 0,  
\]
which shows that $q \notin D_\ast (q)$. 

We wish to apply \rp{Prop::Skoda} to $D_\ast (q)$ with $\all = 1+ \eta/2$, $\mu=1$, $f\equiv 1$, $\Psi(z) = \Psi_q(z) = z-q $ and $\psi= -2\eta \log \dist(\cdot, D_\ast (q))$. Notice that since $q \notin D_\ast (q)$, $\dist(z, D_\ast (q)) \leq |z-q|$ for all $z \in D_\ast (q)$. It follows that
\begin{align*}
  \int_{D_\ast (q)} |f|^2 |\Psi|^{-2\all \mu -2} e^{-\psi} \, dV &= \int_{D_\ast (q)} |z - q|^{-4 - \eta} \blb \dist(z, D_\ast(q)) \brb^{2 \eta} (z) \, dV(z)
  \\ &\leq \int_{D_\ast (q)} |z-q|^{-4+\eta} < C_{D,\eta},   
\end{align*}
where the constant $C_{D,\eta}$ is independent of $q$ and $q_{j_\ast}$.   
Hence there exist functions $h_{\eta,i}(q,\cdot)$ that are holomorphic in $D_\ast (q)$, $\sum_{i=1}^2 h_{\eta,i}(q,z)(z_i -q_i)=1$, and $h_{\eta,i}(q,\cdot)$ satisfies the estimate \re{Dqjast_h_est}. 
\end{proof}
Combining \rp{Prop::Skoda_D_ast} with \rl{Lem::h_est_D_del}, we get  
\begin{prop} \label{Prop::h_est_D_ast}  
 Let $q \in U_0 \sm \ov D$ and suppose that $|q - p| < C r(q)$. For each $\eta>0$, there exist functions $h_{\eta,i}(q,\cdot) \in \Oc(D_\ast(q))$, $i=1,2$ such that 
\[
  h_{\eta,1}(q,z) (z_1 - q_1) + h_{\eta,2}(q,z) (z_2 - q_2) = 1, \quad z \in D_\ast (q), 
\]
and  
\[
  I^\eta_{q} (h) := \int_{D_\ast (q)} \frac{|h_{\eta,i}(q,z)|^2}{|z-q|^2} \blb \dist(z, D_\ast(q)) \brb^{2 \eta} \, dV(z) < C_{D,\eta}.
\]
Furthermore, for $z \in D$ with $|z-p| < c$ the following estimate holds: 
\begin{equation} \label{hq_der_est}  
   |L_1^{\all_1} L_2^{\all_2} h_{\eta,i}(q,z)| 
   \leq C_{D,\all} \frac{I^\eta_{q} (h) }{\blb r(q) + |r(z)| + |z-p| \brb^{\all_1} \blb r(q) + |r(z)| + |g(p,z)| + |z-p|^m \brb^{\all_2+1+\eta} }.     
\end{equation}
In particular, 
\begin{equation} \label{hq_sup_norm} 
   |h_{\eta,i}(q,z)| \leq C_D \frac{I^\eta_{q} (h) }{\blb r(q) + |r(z)| + |g(p,z)| + |z-p|^m \brb^{1+\eta} }.    
\end{equation}
\end{prop}
\begin{proof}
 We take $h_{\eta,i}$ to be the ones constructed in \rp{Prop::Skoda_D_ast}. In view of \re{Dqjast_h_est}, and since $q \notin D_\ast(q)$, $|q-p| < C r(q) < C' \del_{j_\ast}$, we can apply \rl{Lem::h_est_D_del} with $\del = \del_{j_\ast}$. The estimate then follows by the fact that $\del_{j_\ast} \approx r(q_{j_\ast}) \approx r(q)$.  
\end{proof}
Note that the functions $h_{\eta,i}$ come from Skoda's theorem applied to the domain $D_\ast(q)$, and it is not known how the functions $h_{\eta,i}(q,\cdot)$ depend on $q$. To address this issue we apply the trick of Range to show that if $z$ is restricted to a compact subset of $D$, then $h_{\eta,i}$ can be replaced by functions which are smooth in $q \in U_0 \sm D$, while essentially preserving all the relevant properties. In our case, since $q$ does not need to be on the boundary, we need to impose the additional condition that $|q-p| < Cr(q)$, for some fixed $C>1$.  

For $0 < \e < \e_0$, let 
\[
  D'_\e = \{ r_{-\e} <0 \} = \{ r < -\e \} , \quad r_{-\e}:= r(z) + \e.   
\]
To simplify notation we will henceforth fix $\eta >0$ and drop the subscript $\eta$ in $h_{\eta,i}$. For $q \in U_0 \sm D$, define 
\[ 
 \Phi_q(\zeta,z) := \sum_{i=1}^2 h_i (q,z) (z_i-\zeta_i) . 
\] 
Then $\Phi_q \in C^\infty \bl \C^2 \times D_\ast (q) \br \subset C^\infty \bl \C^2 \times \ov D \br $, $h_i(q,\cdot) \in \Oc (D_\ast (q)) \subset \Oc(D) $, and $\Phi_q (q,z) \equiv 1$ on $D_\ast (q)$ and in particular on $D$. 
Since $h_i(q,\cdot)$ is continuous on $D_\ast (q)$ and $D'_\e \Subset D_\ast (q)$, we have $|h_i(q,z)| \leq C_{q,\e}(h)$ for all $z \in \ov{D'_\e}$, where $C_{q,\e}(h)$ is an upper bound of $|h_i (q,\cdot)|$ on $\ov {D'_\e}$. This implies that
\[
  |\Phi_q(\zeta,z) - \Phi_q (q,z) | 
  =\left| \sum_{i=1}^2 h_i(q,z) (q_i - \zeta_i) \right| \leq 2C_{q,\e} (h) |\zeta-q|.  
\]
Hence if we take $|\zeta-q|< \ve_{q,\ve} = \blb 4C_{q,\e} (h) \brb^{-1}$, then $|\Phi_q(\zeta,z) - \Phi_q (q,z) | < 1/2$. In other words, 
\[
  |\Phi_q (\zeta,z)| \geq \yh \quad 
  \text{for} \quad 
  (\zeta,z) \in B(q,\ve_{q,\e}) \times \ov{D_\e'}.  
\]
We now set  
\begin{equation} \label{h_qi_def} 
 h_{q,i}(\zeta,z) = \frac{h_i(q,z)}{\Phi_q (\zeta,z)}, \quad  \Phi_q(\zeta,z) := \sum_{i=1}^2 h_i (q,z) (z_i-\zeta_i).     
\end{equation}
Then $h_{q,i} \in C^\infty \bl B(q,\ve_{q,\e}) \times \ov{D'_\e} \br$, holomorphic in $z \in \ov{D'_{\e}}$, and  
\[
  \sum_{i=1}^2 h_{q,i} (\zeta,z) (\zeta_i-z_i) = 1 
\quad \text{on} \quad 
  B(q,\ve_{q,\e}) \times \ov{D'_\e}.  
\]
By \rp{Prop::h_est_D_ast} and \re{h_qi_def}, we have for $\zeta \in 
B(q,\ve_{q,\e})$ and $z \in \ov{D'_\e}$ with $|z-p|<c$,
\begin{equation} \label{h_qi_init_est}  
   |h_{q,i} (\zeta,z)| \lesssim |h_i(q,z)| 
\leq  \frac{C_{D,\eta}}{\blb r(q) + \e + |r(z)| - \e + |g(p,z)| + |z-p|^m \brb^{1+\eta} }.    
\end{equation}
We need to change $q$ to $\zeta$ on the above right-hand side. For this we need to estimate $\ve_{q,\e} 
:= c_\ast [C_{q,\e}(h)]^{-1}$, where $C_{q,\e}(h)$ is an upper bound of $h_i(q,\cdot)$ on $\ov{D'_\e}$ given in the next equation line, and $c_\ast$ is some small constant to be determined. By \re{hq_sup_norm}, we have
\begin{equation} \label{C_q_del} 
    |h_{\eta,i}(q,z)| \leq \frac{C_{D,\eta}}{ \blb r(q)+ |r(z)| \brb^{1+\eta}} \leq \frac{C_{D,\eta}}{\blb r(q) + \e \brb^{1+\eta} } {} := C_{q,\e}(h), \quad z \in \ov{D'_\e}.   
\end{equation}
Hence if we set $c_\ast < \bl 8 \|r \|_1 \br^{-1}$, then 
\[ 
   \ve_{q,\e} = c_\ast [C_{q,\e}(h)]^{-1} 
  = c_\ast C_{D,\eta}^{-1} \blb r(q) + \e \brb^{1+\eta} 
< \frac{1}{8 \|r \|_1} \bl r(q) + \e \br,    
\] 
where we assume that $C_{D,\eta} > 1$. 
This implies that $|r(\zeta)-r(q)| \leq \|r\|_1 |\zeta-q| < \frac18 (r(q)+\e)$ for $\zeta \in B(q,\ve_{q,\e})$, and consequently 
\[
  r(q) + \e \geq r(\zeta) - |r(\zeta)-r(q)| + \e 
\geq r(\zeta) - \frac{1}{8} r(q) + \frac{7}{8} \e,   
\]
or $(9/8) r(q) + \e \geq r(\zeta) + (7/8) \e$. This further implies that $(9/8) (r(q) + \e) \geq (7/8)(r(\zeta) + \e) $ and 
\begin{equation} \label{rq+del_lb}   
 \yh (r(q) + \e) \geq \yt (r(\zeta) + \e).   
\end{equation} 
By similar reasoning, if we set $ \ve_{q,\e} := c_\ast [C_{q,\e}(h)]^{-1} 
$ with $c_\ast < \blb 8 \| g \|_1 \brb^{-1} $,  $\| g \|_1:= \sup_{|z-p|<a_1} \| g(\cdot,z) \|_1$ we get 
\begin{equation} \label{g_q_zeta_est}  
   |g(q,z)| \geq |g(\zeta,z)| - \frac{1}{8}  \bl r(q) + \e \br, \quad \zeta \in B(q, \ve_{q,\e}). 
\end{equation}

We would like to bound $g(p,z)$ below by  $g(q,z) - \frac18 \bl r(q) + \e \br $. However this would require that $|p-q|$ to be less than $ \bl 8 \| g \|_1 \br^{-1} \bl r(q) + \e \br $ which is clearly impossible since $|p-q| \approx r(q)$ ($r(p) = 0$). To fix this issue, we replace $|g(p,z)|$ by the smaller quantity $|\IM g(p,z)|$. By the proof of \rl{Lem::Im_g} below, $\IM g(\cdot,z)$ can be used as a local coordinate on $bD$. Let $\phi_z:= (r, \IM g(\cdot,z), \cdots) = (\phi_z^1, \phi_z')$ be a local coordinate map and let $p^T$ be the projection of $p$ onto $bD$ under $\phi_z$, i.e. $p^T:= \phi_z^{-1}(0,\phi'_z(\cdot))$. If we choose $|p^T-q^T| < \bl 8 \| g \|_1 \br^{-1} (r(q) + \e)$, then 
\[
  \left| \IM g(p,z) - \IM g(q,z) \right| 
\leq  \| g \|_1 |p^T-q^T| \leq \frac{1}{8} (r(q) + \e),  
\]
which implies $ 
  |\IM g(p,z)| \geq |\IM g(q,z)| - \frac{1}{8} \bl r(q) +\e \br. $
Now \re{g_q_zeta_est} clearly holds with $g$ replaced by $\IM g$. Thus  
\begin{equation} \label{Im_g_lb}  
  |\IM g(p,z)| \geq |\IM g(\zeta,z)| - \frac{1}{4} \bl r(q) +\e \br, \quad \zeta \in B(q, \ve_{q,\e}). 
\end{equation} 
Now, $|z-p| < a_1$, $|p-q| < c_1$ and 
$|q-\zeta| < \ve_{q,\e}$, so there exists $C_1>0$ such that $\| |\cdot-z|^m \|_1 \leq C_m$ on some fixed ball centered at $p$ and contains $z, q, \zeta$ in the given range. 
Hence if we set $\ve_{q,\e} = 
c_\ast [C_{q,\e}(h)]^{-1}$ with $c_\ast < (8 
C_m)^{-1}$, and argue the same way as for \re{g_q_zeta_est}, we get 
\begin{equation} \label{z-q_mpower_lb} 
  |z^T - q^T|^m \geq |z^T -\zeta^T|^m - \frac{1}{8} (r(q) + \e), \quad \zeta \in B(q, \ve_{q,\e}). 
\end{equation} 
Likewise, by choosing $|q^T - p^T| < (8 C_m)^{-1} \bl r(q) + \e \br$, we get 
\[
  \left| |z^T-p^T|^m - |z^T-q^T|^m \right| 
  \leq C_m (8 C_m)^{-1} \bl r(q) + \e \br \leq \frac{1}{8} \bl r(q) + \e \br.     
\]
Thus we have 
\begin{equation} \label{z-p_mpower_lb} 
  |z^T - p^T|^m \geq |z^T - q^T|^m - \frac{1}{8} \bl r(q) + \e \br, \quad |q-p| < C r(q). 
\end{equation} 
Together, \re{z-q_mpower_lb} and \re{z-p_mpower_lb} imply that 
\begin{equation} \label{z-p_mpower_lb_2}  
   |z^T - p^T|^m \geq  |z^T - \zeta^T|^m- \frac{1}{4} \bl r(q) + \e \br, \quad \zeta \in B(q, \ve_{q,\e}), \quad |q^T - p^T| < (8 C_m)^{-1} \bl r(q) + \e \br.  
\end{equation}  
We now fix 
\begin{equation} \label{c_ast} 
 c_\ast = \frac18  \bl \max \{ \| r \|_1, \| g \|_1, C_m \} \br^{-1}, \quad \ve_{q,\e} := c_\ast \bl r(q) + \e \br. 
\end{equation}
Define the set
\[
  U_{p,\e}' := \left\{ q \in U_0 \sm D: |p-q|< C r(q) \quad \text{and} \quad |p^T - q^T| < \ve_{q,\e} \right\},   
\]
and 
\begin{equation} \label{U_p_eps}  
   U_{p,\e} := \{ \zeta \in  B(q, \ve_{q,\e}), \quad q \in U_{p,\e}' \}.    
\end{equation} 
We note that $U_{p,\e}$ contains the ball $B(p, c_\ast \e) $, and for all $\e \to 0$, the set $U'_{p,\e}$ and thus $U_{p,\e}$ contains the ``cone'':  
\[
  \Cc_p := \left\{ q \in U_0 \sm D: |p^T - q^T| < c_\ast \, r(q) \right \}. 
\]  
Using estimates \re{rq+del_lb}, \re{Im_g_lb} and \re{z-p_mpower_lb_2} in \re{h_qi_init_est}, we conclude for each 
$q \in U'_{p,\e}$ and $\zeta \in B_{q,\e}$,   
\[
  |h_{q,i}(\zeta,z)| \leq \frac{C_{D,\eta}}{\blb r(\zeta) + \e + |r(z)| - \e + |g(\zeta,z)| + |z^T - \zeta^T|^m \brb^{1+\eta} }.
\] 

If $\e$ is sufficiently small, $r_{-\e}(z):= r(z) + \e$ is a defining function for the domain $D'_\e$, and 
\begin{equation} \label{-r_del}  
   |r(z)| - \e = -r(z) - \e = - r_{-\e}(z), \quad z \in D'_\e.    
\end{equation}
Thus the above estimate becomes 
\[
  |h_{q,i}(\zeta,z)| \leq \frac{C_{D,\eta} }{\blb \Gm_\e (z,\zeta) \brb^{1+\eta}}, \quad 
  \Gm_\e(z,\zeta) := r_{-\e}(\zeta) - r_{-\e}(z) + |\IM g(\zeta,z)| + |z^T - \zeta^T|^m .  
\]
Next we estimate the $z$-derivatives of $h_{q,i}$. By \re{h_qi_def} and \re{hq_der_est} we have 
\begin{equation} \label{h_qi_der_init_est}
\begin{aligned}
  \left| L_1^{\all_1} L_2^{\all_2} h_{q,i}(\zeta,z) \right| 
\lesssim \left| L_1^{\all_1} L_2^{\all_2} h_i(q,z) \right|  
&\leq \frac{C_{D,\eta,\all}. }{\blb r(q) + |r(z)| + |z-p| \brb^{\all_1} \blb \Gm_\e(z,\zeta)\brb^{\all_2+ 1+\eta} } 
  \\ &= \frac{C_{D,\eta,\all}}{\blb r_{-\e}(q) - r_{-\e}(z) + |z-p| \brb^{\all_1} \blb \Gm_\e(z,\zeta)\brb^{\all_2+ 1+\eta} }, 
\end{aligned}
\end{equation} 
where we used that $r(q) + |r(z)| = r(q) + \e + |r(z)| - \e = r_\e(q) - r_{-\e}(z)$. 
By \re{rq+del_lb} we get $\yh (r(q)+\e) \geq \frac{1}{3} \bl r(\zeta) +\e \br$, 
and similar estimate as in \re{z-q_mpower_lb} shows that $|z^T-q^T| \geq |z^T-\zeta^T| - \frac18 \bl r(q) +\e \br $ for $\zeta \in B(q,\ve_{q,\e})$. 
Since
\[
   \left| \left| z^T - p^T \right| - \left| z^T - q^T \right| \right| \leq |p^T - q^T| < \frac{1}{8} \bl r(q) +\e \br,  \quad q \in U'_{p,\e},  
\]
we get 
\[
  |z^T-p^T| \geq |z^T-q^T| - \frac18 \bl r(q)+\e \br, \quad q \in U'_{p,\e}.  
\]
Consequently for $q \in U'_{p,\e}$ and $\zeta \in B(q, \ve_{q,\e})$, 
\[
  |z^T - p^T| \geq |z^T - \zeta^T| - \frac{1}{4}  \bl r(q) +\e \br.  
\]
It follows from \re{h_qi_der_init_est} that for $z,q,\zeta$ in the chosen ranges, 
\begin{align*}
  \left| D^\all_z h_{q,i}(\zeta,z) \right| 
   \leq \frac{C_{D,\eta,\all}}{\blb r_{-\e}(\zeta) - r_{-\e}(z) + |z^T-\zeta^T| \brb^{\all_1} \blb \Gm_\e(z,\zeta)\brb^{\all_2+ 1+\eta} }
\leq \frac{C_{D, \eta, \all}}{|z-\zeta|^{\all_1}
\blb \Gm_\e(z,\zeta)\brb^{\all_2+ 1+\eta}}.   
\end{align*}

By compactness, finitely many sets $B(q_1,\ve_{q_1,\e}), \dots,B(q_N, \ve_{q_N,\e})$ will cover $\ov {U_{p,\e}}$. By our choice of $c_\ast$ \re{c_ast}, and the fact that all the $q_i$-s are in $U_0 \sm D$, the constants $\ve_{q_i,\e}:= c_\ast [C_{q_i,\e}(h)]^{-1} = c_\ast C^{-1}_{D,\eta} [r(q) + 
\e]^{1+\eta} \geq c_\ast C^{-1}_{D,\eta} \e^{1+\eta}$ are bounded uniformly below by a constant depending only on $D,\eta, r, \e$.      

Choose a partition of unity $\chi_\nu \in C^\infty_c(B(q_\nu,\ve_{q_\nu,\e}))$ with $0 \leq \chi_\nu \leq 1$, for $\nu = 1,\dots, N$ and $\sum_{\nu=1}^N \chi_\nu \equiv 1$ on $\ov U_0 \sm D$. Set 
\[
  h_i^{\e} (\zeta,z) = \sum_{\nu=1}^m \chi_\nu (\zeta) h_{q_\nu,i}(\zeta,z), \quad \text{for $\zeta \in U_{p,\e}$ and $z \in D'_\e$ }.  
\]
We now summarize the properties of $h^\e_i$. 
\begin{prop} \label{Prop::h_del}  
Let $p \in bD$ and fix $\eta>0$. For each sufficiently small $\e>0$ there are functions $h^\e_i \in C^\infty \bl U_{p,\e} \times D'_\e \br$, $i=1,2$, $h_i^\e$ depending on $\eta$, such that the following are true:  \vspace{5pt}
\begin{center} 
\begin{tabular}{@{}cc@{}}
(i)    &  $h^\e_1(\zeta,z) (z_1-\zeta_1) +  h^\e_2(\zeta,z) (z_2-\zeta_2) = 1$; \\[10pt] 
(ii) &  $h^\e_i (\zeta, \cdot) \in \Oc(D'_\e) $, \quad \text{for $\zeta \in U_{p,\e}$};   
\\[10pt]  
(iii) & 
 $ \displaystyle  \left| D^\all_z h_i^{(\e)}(\zeta,z) \right| 
\leq \frac{C_{D, \eta, \all}}{|z-\zeta|^{\all_1}
\blb \Gm_\e(z,\zeta)\brb^{\all_2+ 1+\eta}}, \quad |z-\zeta| <c. $ 
\end{tabular}
\end{center} 
where 
\begin{equation} \label{Gm_e} 
  \Gm_\e(z,\zeta) := r_{-\e}(\zeta) - r_{-\e}(z) + |\IM g(\zeta,z)| + |z^T-\zeta^T|^m.    
\end{equation} 
\end{prop}   

For fixed $\e$, and $z \in D'_\e$, the functions $h_i^{(\e)}(\cdot,z)$ are defined in a neighborhood $U_{p,\e}$ of an arbitrary point $p \in bD$. By compactness of $bD$, we can use a partition of unity in $\zeta$ to patch together the functions $h_i^{(\e)}$ to obtain smooth functions $w_i^{(\e)}(\zeta,z)$ on $(\Uc \sm \ov D) \times D'_\e$, where $\Uc$ is some fixed neighborhood of $\ov D$ that does not shrink with $\e$ (as $\e \to 0$, $U_{p,\e}$ shrinks only in the tangential direction and not in the radial direction, see \re{U_p_eps} and the remark after.) 
Furthermore, $w_i^{(\e)}$ satisfies the same properties (i)-(iii) as $h_i^{(\e)}$ in \rp{Prop::h_del}. We summarize its properties in the following proposition. 
\begin{prop} \label{Prop::leray_map}  
Fix $\eta>0$. For each sufficiently small $\e>0$ there are functions $w_i^{(\e)} \in C^\infty \bl (\Uc \sm \ov D) \times D'_\e \br$, $i=1,2$, $w_i^{(\e)}$ depending on $\eta$, such that the following are true:    \vspace{5pt}  
\begin{center} 
\begin{tabular}{@{}cc@{}}
(i)    &  $w_1^{(\e)}(\zeta,z) (z_1-\zeta_1) +  w^{(\e)}_2(\zeta,z) (z_2-\zeta_2) = 1$; \\[10pt] 
(ii) &  $w^{(\e)}_i (\zeta, \cdot) \in \Oc(D'_\e) $, \quad \text{for $\zeta \in U_{p,\e}$};   
\\[10pt]  
(iii) & 
 $ \displaystyle  \left| D^\all_z w_i^{(\e)}(\zeta,z) \right| 
\leq \frac{C_{D,\eta,\all}}{|z-\zeta|^{\all_1}
\blb \Gm_\e(z,\zeta)\brb^{\all_2+ 1+\eta}}, \quad |z-\zeta| <c. $ 
\end{tabular}
\end{center} 
\end{prop}
We call $w_i^{(\e)}$ the \emph{Leray maps} for the domain $D'_\e$.     

\begin{lemma} \label{Lem::Im_g}    
Let $p_0 \in bD$, one can find a neighborhood $\mc{V}$ of $p_0$ such that for all $z \in \V$ and $\e$ sufficiently small, there exists a coordinate map $\phi_{z}: \mc{V} \to \C^n$ given by $\phi_{z}: \zeta \in \V \to (s, t) = (s_{1}, s_{2},t_1,t_2)$, where $s_1(\zeta) = r_{-\e}(\zeta) \approx \dist(\zeta, b D'_\e)$.   
Moreover, for $z \in \V \cap D'_\e$ and $\zeta \in \V \sm  
\ov D$, the function $\Gm_\e(z,\zeta)$ satisfies 
	\eq{Gm_est} 
	\Gm_\e (z, \zeta) \approx 
 \del_{bD'_\e}(z) + s_{1} + |s_{2}| + |t|^{m} ,  \quad \del_{bD'_\e}(z) := \dist(z, b D'_\e), 
    \eeq
\eq{zeta-z_est}  
  |\zeta -z| \approx \del_{bD'_\e}(z) + s_1 + |s_2| + |t|, 
      \eeq
	for some constant $c$ depending on the domain.  
\end{lemma}
\begin{proof}
 It suffices to show that for fixed $z$ sufficiently close to a boundary point $p_0 \in bD$, the function $\zeta \to \IM g(\zeta,z)$ can be introduced as a (real) local coordinate on $bD$. Let $U_0$ be the neighborhood of $p_0$ such that corresponding to each $p \in U_0$, there exists a biholomorphic map $\phi_p$ given in the beginning of Section 3. 
   We are done if we can show that 
\[
  dr(p_0) \we d_\zeta \IM g(p_0,p_0) \neq 0. 
\]
   Recall that $g(p,z) = \bl \phi_p^{-1} \br^{(2)} (z)$. Since $\phi_p^{-1}(p) = 0$, we have $g(p,p) = 0$ for all $p \in U_0$. Thus $\na_\zeta g(p,p) = - \na_z g(p,p)$. Hence it suffices to show that  
\begin{equation} \label{dr_dz_Im_g_nondeg}  
   dr(p_0) \we d_z \IM g(p_0,p_0) \neq 0.    
\end{equation}  
Without loss of generality, we can assume that $dr(p_0) =dx_2$. Denote $z'= \phi_{p_0}^{-1}(z)$. Then $r(z) = \rho(z') $, and
\[
  \rho(z') = \RE z'_2 + O(|z'|^2), \quad  z_2' = 
  (\phi_{p_0}^{-1})^{(2)} (z) = g(p_0,z).  
\] 
we have  
\[
  \yh = \pp{r}{z_2}(p_0) = \pp{\rho}{z'_2} (0) \pp{z'_2}{z_2}(p_0) = \yh \pp{z'_2}{z_2}(p_0), \quad \text{and} \quad 
  0 =  \pp{r}{z_1}(p_0) = \pp{\rho}{z'_2} (0) \pp{z'_2}{z_1}(p_0) = \yh \pp{z'_2}{z_1}(p_0). 
\]
Thus $\pp{z'_2}{z_2}(p_0) = 1$ and $\pp{z'_2}{z_1}(p_0) = 0$. This implies that 
\[
  d_z \IM g(p_0,p_0) = d_z \IM z_2'(p_0) 
  = \frac{1}{2i} \blb d_z z_2'(p_0) - d_z \ov z'_2(p_0) \brb = \frac{1}{2i} (dz_2 + d \ov z_2) = dy_2.  
\]
Hence $dr(p_0) \we d_z \IM g(p_0,p_0) = dx_2 \we d y_2 \neq 0$. This shows that if $U_0$ is a sufficiently small neighborhood of $p_0$, then for each $z \in U_0$, $\zeta \to \IM g(\zeta,z)$ can be used as a coordinate on $U_0 \cap bD$.   
In fact, by proving \re{dr_dz_Im_g_nondeg} we have also shown that for fixed $\zeta \in U_0$, the function $z \to \IM g(\zeta,z)$ can be introduced as a (real) local coordinate on $U_0 \cap bD$. 
\end{proof}
\vspace{5pt} 
\section{Estimates for the $\db$ homotopy operator} 
The Leray maps $w_i^{(\e)}$ defined on $(\zeta, z) 
\in (\Uc \sm \ov D) \times D'_\e$ (\rp{Prop::leray_map}) allow us to construct a $\db$ homotopy formula on the domain $D'_\e$. Let $W^{(\e)} (z,\zeta): = \bl w_1^{(\e)}(\zeta,z), w_2^{(\e)}(\zeta,z) \br $. We will henceforth fix $\eta>0$ and note that $W^{(\e)}$ depends on $\eta$.

Let $D$ be a pseudoconvex domain of finite type $m$ in $\C^2$. Suppose $\var$ is a $(0,1)$-form such that $\var$ and $\db \var$ are in $H^{s,p}(D)$ for $s>1/p$, or in $\La^s(D) $ for $s>0$, we will show that the following $\db$ homotopy formula
\[ 
  \var = \db \Hc^\e_1 \var + \Hc^\e_2 \db \var 
\]
holds in the sense of distributions on $D'_\e$ for all $\e$ sufficiently small.  
The homotopy operators $\Hc^\e_i$ take the form 
  \begin{equation} \label{Hq_sum}
    \Hc^\e_1 \var =  \Hc_1^0 \var + \Hc_1^1 \var, \quad 
    \Hc^\e_2 \db \var =  \Hc_2^0 \db \var + \Hc_2^1 \db  \var,  
  \end{equation} 
where
	\begin{align} \label{H0H1} 
	\Hc_1^0 \var(z) := \int_{\U} B_{(0,0)} (z,\cdot) \we \Ec \var,  
	\quad  
	\Hc_1^1 \var(z) :=  \int_{\U \sm \ov D}    K_{(0,0)}(z,\cdot) \we [\dbar, \Ec] \var; \\ 
  	\Hc_2^0 \db \var(z) := \int_{\U} B_{(0,1)} (z,\cdot) \we \Ec \db \var,  
	\quad  \label{H_2_db_var}  
	\Hc_2^1 \db \var(z) :=  \int_{\U \sm \ov D } K_{(0,1)}(z,\cdot) \we [\dbar, \Ec] \db \var \equiv 0.  
	\end{align}
Here $B_{(0,q)}$ (resp. $K_{(0,q)}$) stands for the component of $B$ (resp.$K$) of type $(0,q)$ in $z$. $\Ec= \Ec_D$ is an extension operator on $D$ satisfying the boundedness properties in \rc{Cor::Rychkov_ext}, and $\supp \Ec \var \subset \Uc$ for all $\var$. To achieve this we simply multiply the Rychkov extension operator with a smooth cut-off function $\chi$ and denote this new operator by $\Ec$. The kernels are given by
\begin{gather}  \label{B_exp}   
	B(z, \zeta) = \frac{1}{(2 \pi i)^{2}}             \frac{\left<\ov{\zeta} - \ov{z} \, , \, d \zeta \right>}{|\zeta -z |^{2}} \we \left( \dbar_{\zeta,z} \frac{ \left< \ov{\zeta} - \ov{z} \, , \, d \zeta \right>}{|\zeta -z|^{2}} \right), \quad \dbar_{\zeta,z} = \dbar_{\zeta} + \dbar_{z} ; 
 \\ \label{K_exp}  
 K^{(\e)}(z, \zeta) = \frac{1}{(2 \pi i)^2}  \frac{\left<\ov{\zeta} - \ov{z} \, , \, d \zeta \right>}{|\zeta -z |^{2}} \we \frac{\left< W^{(\e)}(z,\zeta), d \zeta  \right>}{\left< W^{(\e)}(z,\zeta), \, \zeta -z \right>} 
    = \frac{1}{(2 \pi i)^2}  \frac{\left<\ov{\zeta} - \ov{z} \, , \, d \zeta \right>}{|\zeta -z |^{2}} \we \left< W^{(\e)}, d \zeta  \right>,  
\end{gather}
where we used that $\left< W^{(\e)}(z,\zeta), \, \zeta -z \right> \equiv 1$ for $(z,\zeta) \in D'_\e \times (\Uc \sm \ov D)$.  

\begin{prop} \label{Prop::H0_est} 
   For any $s>0$  and $1<p<\infty$, the operators $\Hc^0_1$, $\Hc^0_2$ are $H^{s,p}(D) \to H^{s+1,p}(D'_\e)$ bounded and $\La^{s}(D) \to \La^{s+1}(D'_\e)$ bounded. More precisely, there exist constants $C,C'$ depending only on $s$, the dimension $n$, and $\dist(D, \Uc)$ such that 
\begin{gather} \label{H10_var_est} 
   \|\Hc_1^0 \var\|_{H^{s+1,p}(D'_\e)} \leq C \|\var\|_{H^{s,p} (D)}, \quad 
 \| \Hc_1^0 \var\|_{\La^{s+1}(D'_\e)} \leq C' \|\var|_{\La^s(D) }; 
 \\ \label{H20_var_est} 
  \|\Hc_2^0 \db \var\|_{H^{s+1,p}(D'_\e)} \leq C \|\db \var\|_{H^{s,p} (D)}, \quad 
 \| \Hc_2^0 \db \var\|_{\La^{s+1}(D'_\e)} \leq C' \|\db \var|_{\La^s(D) }. 
\end{gather} 
In particular, the constants $C,C'$ are independent of $\e \to 0$.  

\end{prop}
\begin{proof} 
  Let $f$ be a coefficient function of $\Ec \var$ or $\Ec \db \var$. In view of \re{B_exp}, $\Hc^0_1 \var$ and $\Hc^0_2 \db \var$ can be written as a finite linear combination of 
	\[  
	\int_{\Uc} \frac{\ov{\zeta^{i} - z^{i}} }{|\zeta -z|^{4}} f(\zeta) \, dV(\zeta) 
	= \int_{\Uc} \pa_{z_{i}} \left(| \zeta -z|^{-2} \right) 
	f(\zeta) \, dV(\zeta) = c_{0} \pa_{z_{i}} 
( \Gc \ast f)(z), 
	\]
where $\Gc$ denotes the Newtonian potential. The estimates are now standard. We first show that for any non-negative integer $k$ and $0<\all<1$,   
\begin{equation} \label{N_est_non_int} 
 \| \Gc \ast f \|_{W^{k+2+\all,p}(\Uc)} \leq C_{n,k,p} \| f \|_{W^{k+\all,p}(\Uc)},   
\quad 
\| \Gc \ast f \|_{\La^{k+2+\all}(\Uc)} \leq C_n \| f \|_{\La^{k+\all,p}(\Uc)}.    
\end{equation} 
For $k=0$, the reader can refer to \cite[Thm 4.6]{G-T01} for the case of H\"older space, and  \cite[Thm 9.9, p.~230]{G-T01} for the case of Sobolev space. For $k \geq 1$, one can apply integration by parts and use the fact that $f$ is compactly supported in $\Uc$ to move derivatives from the kernel to $f$ (see for example \cite[Prop.~3.2]{Shi21}), we leave the details to the reader. 
Once \re{N_est_non_int} is  established, the general case follows from interpolation (\rp{Prop::opt_interpol}).  
\end{proof}
	\begin{lemma}\label{Lem::st_int_est}   
	Let $\beta \geq 0$, $\all>-1$, and let $0 < \del < \yh $. If $\all < \beta - \frac{1}{m}$, then 
		\[
		\int_0^1 \int_0^1 \int_0^1 \frac{s_1^{\all} t \, ds_1 \, ds_2 \, dt}{(\del + s_1 + s_2 + t^m) ^{2+\beta}  (\del + s_1+ s_2 + t) } \leq C_{\all,\beta}  \del^{\all - \beta + \frac{1}{m}}.  
		\] 
	\end{lemma}
\begin{proof}

	Partition the domain of integration into seven regions: \\ 
	$R_{1}: t> t^m > \del, s_{1}, s_{2}$. We have
	\[
	I \leq \int_{\del^{\frac{1}{m}}}^{1} \frac{t}{t^{2m + m\beta} t} \left( \int_{0}^{t^m} s_{1}^{\all} \, ds_{1} \right) \left( \int_{0}^{t^m} \, ds_{2} \right) \, dt 
	\leq C \int_{\del^{\frac{1}{m}}}^{1} t^{m(\all -\beta)} \, dt \leq C \del^{\all - \beta + \frac{1}{m} }. 
	\] 
	$R_{2}: t> \del > t^m, s_{1}, s_{2}$. We have
	\[ 
	I \leq  \del^{-2 - \beta} \left( \int_{\del}^{\del^{\frac{1}{m}}} \frac{t }{t} \, dt \right) \left( \int_{0}^{\del} s_{1}^{\all} \, ds_{1} \right) \left( \int_{0}^{\del} \, ds_{2} \right)
	\leq C \del^{\all - \beta + \frac{1}{m}}.  
	\] 
	$R_{3}: t> s_{1} > \del, t^m, s_{2}$. We have 
	\[
	I \leq \int_{\del}^{1} \frac{s_{1}^{\all }}{s_{1}^{2 + \beta}} \left( \int_{0}^{s_1^{1/m} } \frac{t}{t} \, dt \right) \left( \int_{0}^{s_{1}} \, ds_{2} \right) \, ds_{1} 
	\leq C \int_{\del}^1 s_1^{\all - \beta + \frac{1}{m} - 1} \, ds_1 
	\leq C \del^{\all- \beta + \frac{1}{m}}.  
	\]
	$R_{4}: t> s_{2} > \del, t^m, s_{1}$. We have
	\[ 
	I \leq \int_{\del}^{1} \frac{1}{s_{2}^{2+\beta} } \left( \int_{0}^{s_2^{1/m}} \frac{t}{t} \, dt \right) 
	\left( \int_{0}^{s_{2}} s_{1}^{\all} \, ds_{1} \right) \, ds_{2}   
	\leq C \int_{\del}^1 s_2^{\all - \beta + \frac{1}{m} - 1} \, ds_2  
	\leq C \del^{\all - \beta + \frac{1}{m}}.  
	\]
	$R_{5}: \del > t, t^m, s_{1}, s_{2}$. We have
	\[ 
	I \leq \del^{-2 -\beta } \del^{-1} \left( \int_{0}^{\del} t \, dt \right) \left( \int_{0}^{\del} s_{1}^{\all} \, ds_{1} \right)    \left( \int_{0}^{\del} \, ds_{2} \right) 
	\leq C \del^{\all - \beta +1}. 
	\] 
	$R_{6}: s_{1} > \del, t, t^m, s_{2}$. We have
	\[ 
	I \leq \int_{\del}^{1} \frac{s_{1}^{\all }}{s_{1}^{2 + \beta} s_{1}} \left( \int_{0}^{s_{1}} t \, dt \right) \left( \int_{0}^{s_{1}} \, ds_{2} \right) \, ds_{1} 
	\leq C  \int_{\del}^1 s_1^{\all - \beta} \, ds_1. 
	\] 
	$R_{7}: s_{2} > \del, t, t^m, s_{1}$. We have
	\[
	I \leq \int_{\del}^{1} \frac{1}{s_2^{2+\beta} s_2} \left( \int_{0}^{s_{2}} t \, dt \right) \left( \int_{0}^{s_{2}} s_{1}^{\all } \, ds_{1} \right)  \, ds_{2}  
	\leq C  \int_{\del}^1 s_2^{\all - \beta} \, ds_2. 
	\]
	Here the constants depend only on $\all$ and $\beta$. For $R_6$ and $R_7$, we have  
	\[   
	\int_\delta^1r^{\alpha-\beta}dr\le \begin{cases}C, &\alpha-\beta >-1, \\C (1+ |\log\del| ), &\alpha-\beta= -1, \\
	C\delta^{\alpha-\beta + 1}, &\alpha-\beta< -1,\end{cases} 
	\] 
	which is bounded by $ C\delta^{\alpha-\beta+\frac{1}{m}}$ in all cases. 
\end{proof}	
We recall that $\Gm_\e$ is given by \re{Gm_e}. 
\begin{lemma}\label{Lem::int_est} 
Let $\beta \geq 0$, $\all > -1$, $\all < \beta - \frac{1}{m}$, and let $0 < \del < \yh $. 
Denote $\del_{bD'_\e}(z):= \dist(z, b D'_\e)$. Then for any $z \in D'_\e$ and $\zeta \in \Uc \sm \ov D$: 
	\[
		\int_{\Uc \sm \ov D} \frac{ 
 \blb \del_{bD'_\e}(\zeta) \brb^{\all}  \, dV (\zeta)}{|\Gm_\e(z, \zeta)|^{2+\beta} |\zeta -z| } \leq C \blb \del_{bD'_\e}(z) \brb^{\all - \beta + \frac{1}{m}}; \quad 
		\int_{D'_\e} \frac{ \blb \del_{bD'_\e}(z) \brb^{\all} \, dV(z)}{|\Gm_\e (z, \zeta)|^{2+\beta} |\zeta -z| } \leq C \blb \del_{bD'_\e}(\zeta) \brb^{\all - \beta + \frac{1}{m}},  
		\] 
	where the constants depends only on $\alpha,\beta$, $D$ and $\Uc$, and is independent of $\e$. 
 \end{lemma}  
 \begin{proof}
 By \rl{Lem::Im_g}, near each $\zeta_0 \in b \Om$ there exists a neighborhood $\V_{\zeta_0}$ of $\zeta_0$ and a coordinate system $\zeta \mapsto ( s =(s_1, s_2), t ) \in \R^2 \times \R^2$ such that
	\begin{align} \label{zeta_coord_est} 
	|\Gm_\e(z, \zeta)| \gtrsim \dist(z, bD'_\e) + |s_1| + |s_2| + |t|^m,  \quad
		|\zeta - z| \gtrsim |(s_1, s_2, t)|  
		\end{align}
		for $z \in \V_{\zeta_0} \cap D'_\e$ and $\zeta \in \V_{\zeta_0} \sm \ov D$, with $|s_1 (\zeta) | \approx \del_{bD'_\e}(\zeta)$. By switching the roles of $z$ and $\zeta$, we can also find a coordinate system $z \mapsto (\ti{s}= (\ti{s}_1, \ti{s}_2), \ti{t}) \in \R^2 \times \R^2$ such that
		\begin{align} \label{z_coord_est} 
		|\Gm_\e(z, \zeta)| \gtrsim \dist(\zeta,bD'_\e) + |\ti{s}_1| + |\ti{s}_2| + |\ti{t}|^2,  \quad
		|\zeta - z| \gtrsim |(\ti{s}_1, \ti{s}_2, t)| 
		\end{align} 
		for $z \in \V_{\zeta_0} \cap D'_\e$ and $\zeta \in \V_{\zeta_0} \sm \ov D$, with $|\ti{s}_1 (z) | \approx \del_{bD'_\e}(z)$. 
	By partition of unity in both $z$ and $\zeta $ variables and \re{zeta_coord_est}, we have
		\begin{align*}
		\int_{\Uc \sm \ov D} \frac{ \blb \del_{bD'_\e}(\zeta)\brb^{\all}  \, dV (\zeta)}{|\Gm_\e(z, \zeta)|^{2+\beta} |\zeta -z| }
		&\leq C_{D} \int_0^1 \int_0^1 \int_0^1 \frac{s_1^{\all} \, t \, ds_1 \, ds_2 \, dt}{(\del_{bD'_\e}(z) + s_1 + s_2 + t^m) ^{2+\beta}  (\del_{bD'_\e}(z) + s_1+ s_2 + t) }. 
		\end{align*} 
By \rl{Lem::st_int_est}, the integral is bounded by $C_{n,\all,\beta} \blb \del_{bD'_\e}(z) \brb^{\all-\beta + \frac{1}{m}}$, which proves the first inequality in the statement of the lemma. The second inequality follows by the same way.     
 \end{proof}
    \begin{prop} \label{Prop::H1_est}   
Fix $\eta>0$. Let $D'_\e \subset \C^2$ be the sequence of domains given by $D'_\e = \{ z \in \Uc: r(z) < -\e \} = \{ z \in \Uc: r_{-\e}(z)<0 \}$. Let $\Hc_1^1 \var$ be given by \eqref{H0H1}, with $\Hc^1_1 \var$ depending on $\eta$. Then the following statements are true. 
\begin{enumerate}[(i)] 
    \item 
For any $1<p<\infty$, $s > \frac{1}{p}$ and non-negative integer $k>s+ \frac{1}{m} - \eta$, there exists a constant $C = C(D,\eta, k,s,p)$ such that for all $\var\in H^{s,p}_{(0,1)}(D'_\e)$,
\begin{equation} \label{wt_Lp_est} 
  \| \del_{b D'_\e}^{k - ( s+\frac{1}{m}-\eta )} D^k \Hc_1^1 \var \|_{L^p(D'_\e)} \leq C \| \var \|_{H^{s,p} (D)},  
\end{equation}
where $\del_{b D'_\e} (z):= \dist(z, b D'_\e)$. 
\item  For any $s > 0$ and non-negative integer $k>s+ 1+ \frac{1}{m} - \eta$, there exists a constant $C = C(D, \eta, k,s)$ such that for all $\var\in \La^s_{(0,1)}(D'_\e)$,
\begin{equation} \label{wt_Linfty_est}  
  \| \del_{b D'_\e}^{k - ( s+\frac{1}{m}-\eta )} D^k \Hc_1^1 \var \|_{L^\infty(D'_\e)} \leq C \| \var \|_{\La^s(D)}.   
\end{equation}
\end{enumerate}
In particular, all the constants C are independent of $\e \to 0$. 
\end{prop} 
\begin{proof} 
(i) We estimate the integral 
\begin{equation} \label{main_int}  
	\int_{D'_\e} \blb \del_{b D'_\e}(z) \brb^{k - (s+ \frac{1}{m} - \eta) p} \left| \int_{\U \sm \ov D}  D_z^k K^{(\e)}(z, \cdot) \we [\dbar, \Ec] \var \right|^p \, dV(z). 
\end{equation}

Let $f$ be a coefficient function of $[\dbar, \Ec]\var$ so that $f \in H^{s-1,p}(\Uc \sm \ov{\Om})$. 
		
We now estimate the inner integral in \re{main_int} which we shall denote by $\Kc f$. By \re{K_exp}, we can  write $\Kc f(z)$ as a linear combination of 
\[
  \int_{\U \sm \ov D} f(\zeta) \blb D_z^\gm W^{(\e)}(\zeta,z) \brb D^{k-\gm}_{z} \left( \frac{\ov{\zeta_i} - \ov z_i }{|\zeta-z|^2} \right) \, dV(\zeta), \quad  0 \leq \gm \leq k. 
\]  
In view of \re{Prop::leray_map} (iii), and the fact that $|z-\zeta| \geq \Gm_\e(z,\zeta)$ for $(z,\zeta) \in D'_\e \times (\Uc \sm \ov D)$ (see \re{Gm_est} and \re{zeta-z_est}), the worst term is when $\gm = k$ and $D_z^\gm = D_{z_2}^k$ Thus  
\begin{equation} \label{Kf_main_term}  
\left| \Kc f(z) \right| \leq C \left| \int_{\U \sm \ov D} f(\zeta) \frac{D_{z_2}^k   W^{(\e)}(z,\zeta) }{|\zeta -z|} \, dV(\zeta) \right| . 
\end{equation}
To simplify notation we will denote the kernel of the above integral by $A(z,\zeta)$. 
By \re{Prop::h_del} (iii), we have 
\begin{equation} \label{B_est} 
  \left| A(z,\zeta) \right|
\leq \frac{C_{D,\eta}}{\blb \Gm_\e(z,\zeta)\brb^{k+1+\eta}|\zeta-z|}.   
\end{equation} 
Let $\mu>0$ be some number to be chosen. We have 
\begin{align*}
  |\Kc f(z)| &\leq C \int_{\U \sm \ov D} |A(z,\zeta)|^{\frac{1}{p}} |A(z,\zeta)|^{\frac{1}{p'}} | f(\zeta)| \, dV(\zeta) 
\\ &= C \int_{\U \sm \ov D} \blb \del_{bD'_\e}(\zeta) \brb^{-\mu} | A(z,\zeta) |^{\frac{1}{p}} \blb \del_{bD'_\e}(\zeta) \brb^{\mu}  |A(z,\zeta) | ^{\frac{1}{p'}} | f(\zeta) | \, dV(\zeta). 
\end{align*}
By H\"older's inequality we get 
\begin{equation} \label{Kf_Holder} 
  | \Kc f (z) |^p \leq C
\left[ \int_{\U \sm \ov D} \blb \del_{bD'_\e}(\zeta) \brb^{- \mu p +\mu} | A(z,\zeta) | 
| f(\zeta)|^p \, dV (\zeta) \right] 
\left[ \int_{\U \sm \ov D } \blb \del_{bD'_\e}(\zeta)\brb^{\mu} | A(z,\zeta)| \, dV(\zeta)  \right]^{\frac{p}{p'}}.     
\end{equation} 
By \re{B_est}, 
\[
  \int_{\U \sm \ov D } \blb \del_{bD'_\e}(\zeta)\brb^{\mu} | A(z,\zeta)| \, dV(\zeta)   
\leq C_{D,\eta}  \int_{\U \sm \ov D } 
\frac{\blb \del_{bD'_\e}(\zeta) \brb^{\mu}}{\blb \Gm_\e(z,\zeta)\brb^{k+1+\eta}|\zeta-z|}  \, dV(\zeta). 
\]
To estimate the integral we apply \rl{Lem::int_est} with $\all = \mu$ and $\beta = k-1+\eta$. 
\begin{equation} \label{B_zeta_int_est}   
   \int_{\U \sm \ov D } \blb \del_{bD'_\e}(\zeta)\brb^{\mu} | A(z,\zeta)| \, dV(\zeta) 
  \leq C_{D,\eta} \blb \del_{bD'_\e}(z) \brb^{\mu-(k-1+\eta)+\frac{1}{m} }. 
\end{equation}
The hypothesis of \rl{Lem::int_est} requires that
we choose 
\begin{equation} \label{ka_range}  
-1 < \mu < k+\eta -1- \frac{1}{m}.  
\end{equation} 
Using \re{B_zeta_int_est} in \re{Kf_Holder} and applying Fubini's theorem, we get 
\begin{equation} \label{Kf_last_int}  
 \begin{aligned} 
   &\int_{D'_\e} \blb \del_{bD'_\e}(z) \brb^{(k-s-\frac{1}{m}+\eta)p} 
|\Kc f(z)|^p \, dV(z) 
\\ &\quad \leq C_{D,\eta} \int_{D'_\e} \blb \del_{bD'_\e}(z) \brb^\si  
\left( \int_{\U \sm \ov D } \blb \del_{bD'_\e}(\zeta) \brb^{- \mu p +\mu} | A(z,\zeta) | 
| f(\zeta)|^p \, dV (\zeta)\right)  \, dV(z) 
\\ &\qquad = C_{D,\eta} \int_{\U \sm \ov D }  \blb \del_{bD'_\e}(\zeta) \brb^{-(p-1)\mu} \left( \int_{D'_\e} \blb \del_{bD'_\e}(z) \brb^\si |A(z,\zeta)|\,dV(z) \right) | f(\zeta)|^p \, dV (\zeta), 
\end{aligned}    
\end{equation}

where we set 
\begin{equation} \label{tau}  
\begin{aligned}
   \si &= \Bl k-s-\frac{1}{m} + \eta \Br p+ \frac{p}{p'} 
  \Bl \mu - \eta - k + 1 + \frac{1}{m} \Br
\\ &= (p-1) \mu + (1-s)p + \Bl  k + \eta - 1 - \frac{1}{m} \Br. 
\end{aligned}
\end{equation} 
By \re{B_est},  
\begin{equation} \label{B_z_int_est} 
  \int_{D'_\e} \blb \del_{bD'_\e}(z)\brb^{\si} 
 |A(z,\zeta)| \, dV(z)  
 \leq C_{D,\eta} \int_{D'_\e} 
\frac{\blb \del_{bD'_\e}(z) \brb^{\si}}{\blb \Gm_\e(z,\zeta)\brb^{k+1+\eta}|\zeta-z|}  \, dV(z).     
\end{equation}

We would like to apply $\all = \si$ and $\beta = k-1+\eta$. For this we need to choose $-1 < \si < k+\eta - 1 - \frac{1}{m}$. In view of \re{tau}, $\mu$ needs to satisfy 
\begin{equation} \label{ka_range_2} 
 \frac{1}{p-1} \Blb p \bl s -1 \, \br - \eta - k + \frac{1}{m} \Brb < \mu < \frac{p}{p-1} (s -1).   
\end{equation}
We need to check that the intersection of \re{ka_range} and \re{ka_range_2} is non-empty. There are two inequalities to check: 
\begin{gather*}
 \frac{p}{p-1} (s -1) > - 1; 
 \\ 
 \frac{1}{p-1} \Blb p \bl s -1 \, \br - \eta - k + \frac{1}{m} \Brb < k + \eta - 1 - \frac1m.  
\end{gather*}
An easy computation shows that the first inequality gives $s > \frac{1}{p}$, while the second inequality reduces to $s<k+\eta+ \frac{1}{p}-\frac{1}{m}$. By letting $k$ be arbitrarily large integers, we see that the admissible range of $s$ is $\bl \frac{1}{p}, \infty \br$.  
Assuming these conditions hold, we can apply \rl{Lem::int_est} to \re{B_z_int_est} and get 
\begin{align*}
  \int_{D'_\e} \blb \del_{bD'_\e}(z)\brb^{\si} 
 |A(z,\zeta)| \, dV(\zeta) 
 \leq C_{D,\eta} \blb \del_{bD'_\e}(\zeta)  \brb^{\si - (k-1+\eta) + \frac{1}{m}}. 
\end{align*}
Substituting the above estimate into \re{Kf_last_int} we obtain 
\begin{align*}
   \int_{D'_\e} \blb \del_{bD'_\e}(z) \brb^{(k-s-\frac{1}{m} + \eta)p} |\Kc f(z)|^p \, dV(z)  
&\leq C_{D,\eta}
\int_{\Uc \sm \ov D } \blb \del_{bD'_\e}(\zeta) \brb^{ -(p-1) \mu + \si - (k-1+\eta) + \frac{1}{m}} | f(\zeta)|^p \, dV(\zeta) 
\\
&= C_{D,\eta}
\int_{\Uc \sm \ov D} \blb \del_{bD'_\e}(\zeta) \brb^{ (1-s) p} | f(\zeta)|^p \, dV(\zeta). 
\end{align*}
Since $s>0$, by \rp{Prop::comm_est} the last expression is bounded by
\begin{align*}
& C_{D,\eta} \int_{\Uc \sm \ov D } \blb \del_{bD'_\e}(\zeta) \brb^{ (1-s)p} \left| [\db,\Ec] \var(\zeta) \right|^p \, dV(\zeta)
\\ &\quad \leq  C_{D,\eta} \int_{\Uc \sm \ov {D'_\e} } \blb \del_{bD'_\e}(\zeta) \brb^{ (1-s)p} \left| [\db,\Ec] \var(\zeta) \right|^p \, dV(\zeta)
  \leq C_{D,\eta, s}  |\var|_{H^{s,p}(D'_\e)}     
\end{align*}
\\[10pt]  
(ii) We follow the same notation as in (i). Let $\Kc$ be the inner integral in \re{main_int}. By \re{Kf_main_term} and \re{B_est}, we get 
\[
  |\Kc f(z)| \leq C_{D,\eta} \int_{\Uc \sm \ov D } \frac{ | f(\zeta)|}{\blb \Gm_\e(z,\zeta)\brb^{k+1+\eta}|\zeta-z|} 
 \, dV(\zeta), 
\]
where $f$ is a cofficient function of $[\db,\Ec]\var$. 
Now by \rp{Prop::comm_est} applied to the domain $D'_\e$, we have
\[
\bn \blb \del_{bD'_\e}(\zeta) \brb^{1-s} f \bn_{L^\infty(\Uc \sm \ov D)} \leq 
  \bn \blb \del_{bD'_\e}(\zeta) \brb^{1-s} f \bn_{L^\infty(\Uc \sm \ov{D'_\e})} \leq C_s |\var|_{\La^s(D'_\e)}, \quad s>0. 
\]
Hence the above integral is bounded by 
\begin{align*}
   |\Kc f(z)| &\leq C_{D,\eta,s} |\var|_{\La^s(D'_\e)} \int_{\Uc \sm \ov D} \frac{\blb \del_{bD'_\e}(\zeta) \brb^{s-1}}{\blb \Gm_\e(z,\zeta)\brb^{k+1+\eta}|\zeta-z|} \, dV(\zeta). 
\end{align*} 
Let $\all = s-1$ and $\beta= k-1+\eta$. The condition $-1 < \all < \beta - \frac{1}{m}$ in \rl{Lem::int_est} becomes 
\[ 
0 <s < k + \eta - 1 - \frac{1}{m}.
\]
Assuming the above line holds, we can apply \rl{Lem::int_est} to get 
\[
  |\Kc f (z)| \leq C_{D,\eta,s} |\var|_{\La^s(D'_\e)} 
\blb \del_{bD'_\e}(z) \brb^{s-k-\eta+\frac{1}{m}}, \quad z \in D'_\e  .  
\]
In other words, $\sup_{z \in \ov{D'_\e}} \blb \del_{bD'_\e}(z) \brb^{k- \bl s+\frac{1}{m}-\eta \br} \left| D^k \Hc^1 \var(z) \right| \leq C_{D,\eta, s}  |\var|_{\La^s(D'_\e)} $.     
\end{proof} 
We now combine \rp{Prop::H0_est} and \rp{Prop::H1_est} and \rl{Prop::H-L} to obtain the following estimates for the homotopy operator on $D'_\e$. 
\begin{thm} \label{Thm::EstHq}
Fix $\eta>0$. Let $D \subset\C^2$ be a smoothly bounded pseudoconvex domain of finite type $m$, with the defining function $r$. Let $\Hc_1^{(\e)} \var = \Hc^0_1 \var + \Hc^1_1 \var$, $\Hc_2^{(\e)} \db \var = \Hc_2^0 \db \var$ where $\Hc^0_1, \Hc^1_1, \Hc_2^0$ are given by \eqref{H0H1} and \eqref{H_2_db_var}, and $\Hc_1^{(\e)}$ depends on $\eta$. Then $\Hc_1^{(\e)}$ and $\Hc_2^{(\e)}$ are bounded on the following spaces:
\begin{enumerate}[(i)] 
    \item 
    $\Hc_1^{(\e)}: H_{(0,1)}^{s,p} (D)\to H^{s+\frac{1}{m}-\eta,p}(D'_\e)$, for any $1<p<\infty$ and $s>\frac1p$. 
    \item 
   $\Hc_1^{(\e)}: \La^s_{(0,1)}(D) \to \La^{s+\frac{1}{m}-\eta}(D'_\e)$ for any $s>0$. 
    \item 
    $\Hc_2^{(\e)}: H_{(0,2)}^{s,p} (D)\to H^{s+1,p}_{(0,1)}(D'_\e)$, for any $1<p<\infty$ and $s > \frac{1}{p}$.   
    \item 
   $\Hc_2^{(\e)}: \La^s_{(0,2)}(D) \to \La_{(0,1)}^{s+1}(D'_\e)$ for any $s>0$. 
\end{enumerate}
More specifically, there exist constants $C, C', C'', C'''$ depending on $D,\eta,s$ and independent of $\e$ such that
\begin{gather} \label{H1_e_var_est}   
  |\Hc_1^{(\e)} \var|_{H^{s+\frac{1}{m}-\eta,p}(D'_\e)} \leq C |\var|_{H^{s,p}(D)}, \quad 
  |\Hc_1^{(\e)} \var|_{\La^{s+\frac{1}{m}-\eta}(D'_\e)} \leq C' |\var|_{\La^s(D)}; 
  \\ \label{H2_e_var_est} 
 |\Hc_2^{(\e)} \db \var|_{H^{s+1,p}(D'_\e)} \leq C'' |\db \var|_{H^{s,p}(D)}, \quad 
  |\Hc_2^{(\e)} \db \var|_{\La^{s+1}(D'_\e)} \leq C''' |\db \var|_{\La^s(D)}.  
\end{gather}
\end{thm}
\begin{proof} 
The first statement in \re{H1_e_var_est} follows from \re{H10_var_est}, \re{H-L_Sob_est}, and \re{wt_Lp_est}; the second statement follows from  \re{H10_var_est}, \re{H-L_HZ_est}, and \re{wt_Linfty_est}. Estimates  \re{H2_e_var_est} follows from \re{H20_var_est}. 
\end{proof}
\begin{cor}[Homotopy formula] \label{Cor::hf} 
		Under the assumptions of \rt{Thm::EstHq}, suppose
  either 
  \\[5pt] 
  a) $\var\in H^{s,p}_{(0,1)}(D)$ and $\dbar\var\in H^{s,p}_{(0,2)}(D)$, for $1<p<\infty$ and $s>\frac1p$; or   
  \\ b) $\var\in \La^s_{(0,1)}(D)$ and $\dbar \var\in \La^{s}_{(0,2)}(D)$, for $s>0$. 
  \\[5pt] 
  Then for all sufficiently small $\e>0$, the following homotopy formula holds in the sense of distributions: 
	\begin{equation}\label{hf_eqn}
	\varphi=\dbar \Hc_1^{(\e)} \varphi+ \Hc_2 ^{(\e)}\dbar\varphi,
  \quad \text{on $D'_\e$}. 
		\end{equation}
	In particular if $\varphi \in H^{s,p}_{(0,1)}(D)$ (resp. $\var \in \La^s_{(0,1)}(D)$) is $\dbar$-closed, then we have $\dbar \Hc_1^{(\e)} \var = \var$ on $D'_\e$ and $\Hc_1^{(\e)} \var \in H^{s+\frac1m-\eta,p}(D'_\e)$ (resp.  
 $\Hc_1^{(\e)} \var \in \La^{s+\frac{1}{m}-\eta}(D'_\e)$), with $\Hc_1^{(\e)}$ satisfying estimates \eqref{H1_e_var_est}.  
	\end{cor}  
	\begin{proof}
 By \cite[Prop.~2.1]{Gong19}, Formula \eqref{hf_eqn} is valid for $\varphi \in C^\infty_{(0,1)} (\ov{D'_\e})$.  
For general $\var \in H^{s,p}_{(0,1)}(D)$ such that $\dbar\varphi\in H^{s,p}_{(0,2)}(D)$ and with $s,p$ in the given range, we use approximation. By \cite[Prop.~A.3]{S-Y24_1}, there exists a sequence $\var_\la\in C^{\infty}(\ov D)$ such that 
		\begin{gather*}
		\var_\la \ra \var \quad \text{in $H^{s,p}(D)$ },  
		\\ 
		\dbar \var_\la \ra  \dbar \var \quad \text{in $H^{s,p}(D)$}. 
		\end{gather*} 
		By \rt{Thm::EstHq}, we have for $1<p<
        \infty$ and $s>\frac{1}{p}$,  
		\begin{align*} 
		\| \dbar  \Hc_1^{(\e)} (\var_\la - \var) \|_{H^{s+\frac{1}{m}-\eta-1, p} (D'_\e)} 
		&\leq \| \Hc_1^{(\e)} (\var_\la - \var) \|_{H^{s+\frac{1}{m}-\eta, p} (D'_\e)}  
		\\ &\leq C \| \var_\la -  \var \|_{H^{s,p}(D)}, 
		\end{align*}
		and also  
		\[
		\| \Hc_2^{(\e)} \dbar (\var_\la - \var) \|_{H^{s+1, p}(D'_\e)}  
		\leq C' \| \dbar (\var_\la- \var)  \|_{H^{s,p} (D)}.  
		\]
		Letting $\la \to \infty$ we obtain \re{hf_eqn}. 
        The proof for the H\"older space is similar and we leave it to the reader. 
	\end{proof} 

We can now use \rc{Cor::hf} to prove \rt{Thm::main}. 
\\[5pt]
\textit{Proof of \rt{Thm::main}}. 
Let $\{ \e_j \}$ be a sequence of small positive numbers tending to $0$. Consider the sequence of functions $\{ \Ec_j \Hc_1^{(\e_j)} \var \}$, and $ \{ \Ec_j \Hc_2^{(\e_j)} \db \var \}$, $\Ec_j$ being the Rychkov extension operator on $D'_{\e_j}$. Then by \rc{Cor::Rychkov_ext} and \rt{Thm::EstHq}, we have
\begin{gather*}
 |\Ec_j \Hc_1^{(\e_j)} \var|_{H^{s+\frac{1}{m}-\eta,p}(\C^n)} \leq C_{s,\eta} |\Hc_1^{(\e_j)} \var|_{H^{s+\frac{1}{m}-\eta,p}(D'_{\e_j})} \leq C_{D,\eta,s} |\var|_{H^{s,p}(D)}; 
 \\ 
  |\Ec_j \Hc_2^{(\e_j)} \db \var|_{H^{s+1,p}(\C^n)} \leq C_{s,\eta} |\Hc_2^{(\e_j)} \db \var|_{H^{s+1,p}(D'_{\e_j})} \leq C_{D,\eta,s} |\db \var|_{H^{s,p}(D)}. 
\end{gather*}
By the Banach-Alaouglu theorem, there exists a subsequence $\{ \Ec_{j_\all} \Hc_1^{(\e_{j_\all})} \var \}$ which converges weakly to some limit function $\Hc_1 \var$ in $H^{s+\frac{1}{m}-\eta,p}(\C^n)$. 
Similarly, by extracting another subsequence from $\{ \Ec_{j_\all} \Hc_2^{(\e_{j_\all})} \db \var \}$, and denoting this subsequence still by $\{ \Ec_j \Hc_2^{(\e_j)} \db \var \}$, we get $\Ec_j \Hc_2^{(\e_j)} \db \var$ converges weakly to a limit function $\Hc_2 \db \var$. 

Now, by \rc{Cor::hf}, we have 
$\var = \db \Hc_1^{(\e_j)} \var + \Hc_2^{(\e_j)} \db \var$ on $D'_{\e_j}$.     
Let $\phi = \phi_1 d \hht {\ov z_1} + \phi_2 d\hht { \ov z_2}$ be a $(2,1)$ form with coefficients in $C^\infty_c(D)$. Here $d \hht{\ov z_1}:= d \ov z_2 \we d z_1 \we d z_2 $ and 
$d \hht{\ov z_2}:= d \ov z_1 \we d z_1 \we d z_2 $. 
Then there exists a large $J$ such that $\supp \phi \subset D'_{\e_j}$, for all $j > J$. It follows that for $j > J$: 
\begin{align*} 
   \left< \var, \phi \right>_D 
  = \left< \var, \phi \right>_{D'_{\e_j}} 
&= \left< \db \Ec_j \Hc_1^{(\e_j)} \var + \Ec_j \Hc_2^{\e_j} \db \var, \phi \right>_{D'_{\e_j}}
= \left< \db \Ec_j \Hc_1^{(\e_j)} \var, \phi \right>_{D'_{\e_j}} + \left< \Ec_j \Hc_2^{\e_j} \db \var, \phi \right>_{D'_{\e_j}}
\\&= - \left< \Ec_j \Hc_1^{(\e_j)} \var, \db \phi \right>_{D'_{\e_j}} + \left< \Ec_j \Hc_2^{\e_j} \db \var, \phi \right>_{D'_{\e_j}}
= - \left< \Ec_j \Hc_1^{(\e_j)} \var, \db \phi \right>_{D} + \left< \Ec_j \Hc_2^{\e_j} \db \var, \phi \right>_{D}.    
\end{align*} 
Taking limit as $j \to \infty$, and using that 
\[ 
 \lim_{j\to \infty} - \left< \Ec_j \Hc_1^{(\e_j)} \var, \db \phi \right>_{D} = - \left< \Hc_1 \var, \db \phi \right>_D, \quad 
  \lim_{j\to \infty} \left< \Ec_j \Hc_2^{\e_j} \db \var, \phi \right>_{D} =  \left< \Hc_2 \db \var, \phi \right>_{D}  , 
\] 
We get
\[
  \left< \var, \phi \right>_D  
  = - \left< \Hc_1 \var, \db \phi \right>_{D} 
+ \left< \Hc_2 \db \var, \phi \right>_{D} 
  = \left< \db \Hc_1 \var + \Hc_2 \db \var, \phi \right>_{D}.  
\] 
In other words, $\var = \db \Hc_1 \var + \Hc_2 \db \var$ in the sense of distributions on $D$. 

Next we prove the statement for the H\"older-Zygmund space. Let $\e_j$ and $\Ec_j$ be as above. The sequence $\Ec_j \Hc^{(\e_j)} \var$ satisfies the estimate 
\begin{gather*}
   |\Ec_j \Hc_1^{(\e_j)} \var|_{\La^{s+\frac{1}{m}-\eta}(\C^n)} \leq C_s |\Hc_1^{(\e_j)} \var|_{\La^{s+\frac{1}{m}-\eta}(D'_{\e_j})} \leq C_{D,\eta,s} |\var|_{\La^s(D) }; 
   \\ 
|\Ec_j \Hc_2^{(\e_j)} \db \var|_{\La^{s+1}(\C^n)} \leq C_s |\Hc_2^{(\e_j)} \db \var|_{\La^{s+1}(D'_{\e_j)}} \leq C_{D,\eta,s} |\db \var|_{\La^s(D) }. 
\end{gather*}
In particular, $\{ \Ec_j \Hc_1^{(\e_j)} \var \}, 
\{ \Ec_j \Hc_2^{(\e_j)} \db \var\} $ are families of equicontinuous functions on $\ov D$. By the Ascoli-Arzela theorem and by taking two subsequence successively, we may assume that 
$\Ec_j \Hc_1^{(\e_j)} \var$ and $\Ec_j \Hc_2^{(\e_j)} \db \var$ converge uniformly to some limit functions $\Hc_1 \var, \Hc_2 \db \var$ respectively, which are in $\La^{s+\frac{1}{m}-\eta}(D)$. 
Furthermore, we have $\var = \db \Hc_1 \var + \Hc_2 \db \var$ in the sense of distributions on $D$. The proof is now complete. \qedhere 

\bibliographystyle{amsalpha}
\bibliography{Reference}  

\end{document}